\documentclass[12pt]{article}
\usepackage[margin=1in]{geometry}
\usepackage{amsmath}
\usepackage{amsthm}
\usepackage{amssymb}
\usepackage{color}
\usepackage[utf8]{inputenc}
\usepackage[english]{babel}
\usepackage{enumitem}
\setlist[itemize]{leftmargin=0.2pt}
\usepackage{graphicx}
\usepackage{natbib}
\usepackage{url}
\usepackage{hyperref}
\usepackage[toc,page]{appendix} 
\usepackage{tabularx}
\usepackage{subcaption}
\usepackage{setspace} 
\usepackage{authblk}
\usepackage{xr}
\usepackage[ruled,vlined]{algorithm2e}
\usepackage{ragged2e}

\usepackage{booktabs}
\usepackage{multirow}

\DontPrintSemicolon

\usepackage{enumitem} 
\newlist{condenum}{enumerate}{1} 
\setlist[condenum]{label=\bfseries Condition C\arabic*., 
                   ref=\arabic*, wide}

\newtheorem{theorem}{Theorem}

\newtheorem{lemma}[theorem]{Lemma}

\newtheorem{remark}{Remark}

\newcounter{numexample}[section]
\newenvironment{numexample}[1][]{\refstepcounter{numexample}\par\medskip
   \noindent \textbf{Numerical Example~\thenumexample. #1} \rmfamily}{}

\newcommand{\RNum}[1]{\uppercase\expandafter{\romannumeral #1\relax}}

\usepackage{xpatch}
\makeatletter
\AtBeginDocument{\xpatchcmd{\@thm}{\thm@headpunct{.}}{\thm@headpunct{}}{}{}}
\makeatother

\title{A Note on the Likelihood Ratio Test in\\ High-Dimensional Exploratory Factor Analysis\footnotetext{This research is partially supported by NSF CAREER SES-1846747, DMS-1712717, and SES-1659328.}}

\author{Yinqiu He, Zi Wang, and Gongjun Xu}
\affil[]{\small{Department of Statistics, University of Michigan}}
\date{}

\providecommand{\keywords}[1]
{
  \small	
  \textit{Keywords:} #1
}

\begin{document}

\maketitle

 \doublespacing


\begin{abstract}
The likelihood ratio test is widely used in  exploratory factor analysis to assess the model fit and determine the number of latent factors. Despite its popularity and clear statistical rationale, researchers have found that when the dimension of the response data is large compared to the sample size, the classical chi-square approximation of the  likelihood ratio test statistic often fails.  Theoretically, it has been an open problem when such a phenomenon happens as the dimension of data increases;
practically, the effect of high dimensionality is less examined in exploratory factor analysis, and there lacks a clear statistical guideline on the validity of the conventional chi-square approximation. 
To address this problem,  we investigate  the failure of the chi-square approximation of the likelihood ratio test in high-dimensional exploratory factor analysis, and derive the {\it necessary and sufficient} condition  to ensure the validity of the chi-square approximation. The results yield simple quantitative guidelines to check in practice and would also provide useful statistical insights into the practice of exploratory factor analysis.

\end{abstract}


\newpage 
\section{Introduction}

Exploratory factor analysis serves as a popular statistical tool to gain insights into latent structures underlying the observed data \citep{Gorsuch1988,fabrigar2011exploratory,bartholomew2011latent}. It is widely used in many application areas such as psychological and social sciences  \citep{Fabrigar1999,preacher2002exploratory,thompson2004exploratory,Finch2016}. In factor analysis, the relationship  among  observed variables in data are explained by a smaller number of unobserved underlying variables, called common factors. To understand the underlying scientific patterns, one fundamental problem in factor analysis is to decide the minimum number of latent common factors that is needed to describe the statistical dependencies in data.

In order to determine the number of factors in exploratory factor analysis, a wide variety of procedures have been proposed; see reviews and discussions in \cite{costello2005best}, \cite{barendse2015using} and \cite{luoExploratory2019}. 
For instance, one broad class of criteria are based on the eigenvalues of the sample correlation matrix of the observed data. Examples include the  Kaiser criterion \citep{kaiser1960application}, the scree test \citep{Cattell1966}, the parallel analysis method \citep{Horn1965,keeling2000regression,dobriban2017permutation}, and  testing linear trend of eigenvalues \citep{bentler1998tests} among many others. 
Another class of methods propose various goodness-of-fit indexes to  select the number of factors,  such as AIC \citep{akaike87factor}, BIC \citep{schwarz1978estimating}, the reliability coefficient \citep{Tucker1973}, and the root mean square error of approximation \citep{steiger2016notes}. Moreover, 
the likelihood ratio test  provides another popularly used approach in practice  \citep{bartlett1950tests,anderson1958introduction}.

Among the various criteria to determine the number of factors, the likelihood ratio test  plays a  unique role, as it is based on a formal  hypothesis testing procedure with a clear statistical rationale and also has a solid theoretical foundation with guaranteed  statistical properties.
In particular, the likelihood ratio test examines how a factor analysis model fits the data using a hypothesis testing framework based on the likelihood theory. 
The classical statistical theory shows that under the null hypothesis,  the likelihood ratio test statistic (after proper scaling)  asymptotically converges to a chi-square distribution, with the degrees of freedom equal to the difference in the number of free parameters between the null and alternative hypothesis  models  
  \cite[see, e.g.,][Section 14.3.2]{anderson1958introduction}. 

In the modern big data era, it is of emerging interest to analyze high-dimensional data \citep{Finch2016,harlow2016big,chen2019joint},  
where throughout this paper we refer to the dimension of the observed response variables as the dimension of data. 
Classical asymptotic theory, despite its importance, often  replies on the assumption that the data dimension is fixed as the sample size increases. Such an assumption often fails in high-dimensional data analysis with large data dimension, and therefore the corresponding asymptotic theory is no longer  directly applicable to modern  high-dimensional applications. In fact, it has been found in the recent statistical literature  that the chi-square approximations for the likelihood ratio test statistics   can become inaccurate as the dimension of data increases with the sample size  \citep[e.g.][]{bai2009,jiang2013,He2018}. 
In factor analysis, although considerable high-dimensional statistical analysis results have been recently developed  \citep{bai2002determining,bai2012statistical,sundberg2016exploratory,ait2017using,chen2020determining}, 
less attention has been paid to the statistical properties of the popular likelihood ratio test under high dimensions. 
Particularly, it remains an open problem  when the conventional chi-square approximation of  the likelihood ratio test starts to fail  as the data dimension grows. In other words, for a dataset with sample size $N$, how large the data dimension $p$ can be to still ensure the validity of the chi-square approximation of the likelihood ratio test?

To better understand this issue, 
this paper investigates the influence of the data dimensionality on the likelihood ratio test in high-dimensional exploratory factor analysis. 
Specifically,  under the null hypothesis, we derive the  necessary and sufficient condition for the chi-square approximation  to hold.
The results consider both the likelihood ratio test without and with the Bartlett correction, and   provide useful  quantitative guidelines that are easy to check in practice. 
Our simulation results are consistent with the theoretical conclusions, suggesting good finite-sample performance of the developed theory. 

The rest of the paper is organized as follows. 
In Section \ref{sec2.1}, we give a brief review of the exploratory factor analysis and the likelihood ratio test, and in Section \ref{sec2.2}, we present our theoretical and numerical results on the performance of the chi-square approximation under high dimensions. Several extensions are discussed in Section 3, and the technical proofs and additional simulation studies are deferred to the appendix.

\section{Likelihood Ratio Test under High Dimensions} \label{sec2}


\subsection{Likelihood ratio test for exploratory factor analysis} \label{sec2.1}
In this section, we briefly review   the likelihood ratio test in exploratory factor analysis  \citep[see, e.g.,][Section 14]{anderson1958introduction}. 
Suppose $X_i$, $i=1,\ldots, N$ are independent and identically distributed $p$-dimensional random vectors. The exploratory factor analysis considers the following common-factor model
\begin{eqnarray}\label{eq00}
X_i = \mu + \Lambda F_i + U_i,
\end{eqnarray} where $\mu$ a the $p$-dimensional mean parameter vector, $\Lambda$ is a $p \times k_0$ loading matrix with $\mathrm{rank}(\Lambda)=k_0<p$, $F_i$ is a $k_0$-dimensional random vector containing the common factors, and $U_i$ is a  $p$-dimensional error vector.   
It is well known that the factor model \eqref{eq00} is not
identifiable without additional constraints, and there are many ways to impose identifiablity  restrictions \citep{anderson1958introduction,bai2012statistical}. In this paper, we focus on the following  identification conditions which have been popularly used in exploratory factor analysis.
In particular, we assume that $F_i$ and $U_i$ are independent latent random vectors with $\mathrm{E}(F_i)=\mathbf{0}_{k_0}$, $\mathrm{cov}(F_i)=\mathrm{I}_{k_0}$, $\mathrm{E}(U_i)=\mathbf{0}_p$, and $\mathrm{cov}(U_i)=\Psi$, where $\mathbf{0}_{k_0}$ denotes a $k_0$-dimensional all-zero vector, $\mathrm{I}_{k_0}$ represents a $k_0\times k_0$ identity matrix, and $\Psi$ is a $p\times p$ diagonal matrix with $\mathrm{rank}(\Psi)=p$.
It follows that the population covariance matrix $\Sigma=\mathrm{cov}(X_i)$ can be expressed as
\begin{eqnarray}
\Sigma = \Lambda\Lambda^\top + \Psi. \label{eq01}
\end{eqnarray}

Typically, the true number of common factors $k_0$ is  unknown. In exploratory factor analysis, to determine the number of factors  in model \eqref{eq00},  various  procedures have been developed.  Among them, the likelihood ratio test plays a  unique role due to its solid theoretical foundation and nice statistical properties. 
The common practice utilizes the model's likelihood function  assuming both  $F_i$ and $U_i$  to be normally distributed. 
In such case, $X_i$ follows a multivariate normal distribution with mean vector $\mathbf{0}_p$ and covariance matrix $\Sigma$ as in \eqref{eq01}, and we write  $X_i\sim  \mathcal{N}(\mathbf{0}_p, \Sigma)$.
 Then, the likelihood ratio test is used to sequentially  test the factor analysis model with a specified number of factors against the saturated model \citep[e.g.,][]{Hayashi2007}. Specifically, for each  $k = 0, 1, \ldots, p$, we consider the following null and alternative hypotheses:
$$\mbox{$H_{0,k}:$  $\Sigma = \Lambda\Lambda^\top + \Psi$ 
    with (at most)  $k$ factors, \, versus~ $H_{A,k}:$ $\Sigma$ is any positive definite matrix.}$$
In practice without a priori knowledge, a typical procedure examines the above hypotheses in a forward stepwise manner. Specifically, we first consider $k=0$ and examine $H_{0,0}: k_0=0$ versus $H_{A,0}$ using the likelihood ratio test, that is, testing whether there is any factor in model \eqref{eq00}.
If $H_{0,0}$ is rejected, we then consider $k=1$,  that is,   a 1-factor model in   the null hypothesis $H_{0,1}$. If  $H_{0,1}$ is rejected, we proceed with $k=2$, and test a 2-factor model for $H_{0,2}$. 
This testing procedure continues until we fail to reject $H_{0,\hat{k}}$ for some $\hat{k}$. Then $\hat{k}$ is taken as an estimate of the true number of factors based on the likelihood ratio test.   



We next introduce the details on the abovementioned  likelihood ratio test. 
For $k=0,$ $H_{0,0}$ examines the existence of any significant factors, which is an important problem in psychology applications \citep[e.g.,][]{Mukherjee1970}. 
This test can be written as 
$$H_{0} : \Sigma = \Psi  ~ \mbox{  versus  } ~ \ H_{A} : \Sigma \neq \Psi,$$
that is, testing whether $ \Sigma$ is a diagonal matrix.
Statistically,  this is also equivalent to the following hypothesis test  
$$H_{0} : R = \mathrm{I}_{p} \ ~ \mbox{  versus  } ~\ H_{A} : R \neq \mathrm{I}_{p},$$ 
where $R$ denotes the population correlation matrix of the response variables $\{X_i, i=1,\ldots, N\}$. 
Under the normality assumption of $X$, $H_{0,0}$ then tests  for the complete independence between $p$ dimensions of  $X$. 
The likelihood ratio test statistic for $H_{0,0}$ with the chi-square limit is
$T_0=-(N-1)\log( |\hat{R}_N|)$,  
where 
$\hat{R}_N$ denotes the sample correlation matrix of the observations $\{X_i, i=1,\ldots, N\}$, and $|\hat{R}_N|$ denotes the  determinant of $\hat{R}_N$; see, e.g., \cite{bartlett1950tests}.  
When the dimension $p$ is fixed and the sample size $N\to \infty$, under $H_{0,0}$, 
\begin{equation}\label{eq1}
	T_0 \xrightarrow{D} \chi^2_{f_0}, \quad \text{with}\  f_0 = p(p-1)/2,
\end{equation} where $\xrightarrow{D}$ represents the convergence in distribution, and $\chi^2_{f_0}$ represents a random variable following the chi-square distribution with degrees of freedom $f_0$. To improve the finite-sample performance, researchers have proposed using the Bartlett correction for the likelihood ratio test   \citep{bartlett1950tests}. The corrected test statistic is 
 $\rho_0  T_{0}$ with the Bartlett correction term $\rho_0 = 1- ({2p+5})/\{6(N-1)\}$, and
under $H_{0,0}$ with  fixed $p$ and $N\to \infty$, we still have the chi-square approximation: 
\begin{eqnarray}
	\rho_0 \times T_{0} \xrightarrow{D} \chi^2_{f_0} , \label{eq2}
\end{eqnarray} while it
improves the convergence rate of the chi-square approximation (3) from $O(N^{-1})$ to $O(N^{-2})$.

For $k \geq 1$, $H_{0,k}$ examines whether the $k$-factor model fits the observed data. 
Under the $k$-factor model,  
let $\hat{\Lambda}_k$ and $\hat{\Psi}_k$ denote the maximum likelihood estimators of $\Lambda$ and $\Psi$, respectively, and define $\hat{\Sigma}_k = \hat{\Lambda}_k \hat{\Lambda}_k^{\top} + \hat{\Psi}_k$. 
Then to test $H_{0,k}$,  
the  likelihood ratio test statistic can be written as
\begin{align}
T_{k} = -(N-1)\log( |\hat{\Sigma}|\times |\hat{\Sigma}_k|^{-1} )+ (N-1) \{\mathrm{tr}(\hat{\Sigma} \hat{\Sigma}_k^{-1})-p \},\label{eq:tkstat}	
\end{align}
where $\hat{\Sigma}$ is the unbiased sample covariance matrix of the observations $\{X_i, i=1,\cdots, N\}$, and $\mathrm{tr}(A)$ denotes the trace of a matrix $A$; see, e.g., \cite{lawley1962factor}. 
Under  the null hypothesis  with  $k_0=k$,   $p$ fixed and $N\to \infty$,  we have the following chi-square approximation:  
\begin{eqnarray}\label{eq3}
	T_k \xrightarrow{D}  \chi^2_{f_k},\quad \text{where }\ f_k= \{ (p-k)^2 - p -k\}/2. 
\end{eqnarray} 
Moreover,  applying the Bartlett correction for this test, we have
\begin{eqnarray}\label{eq4}
	\rho_k \times T_k \xrightarrow{D} \chi_{f_k}^2, \quad \text{where}\ \rho_k = 1-\frac{2p+5+4k}{6(N-1)}. 
\end{eqnarray}



Despite the usefulness of the above chi-square approximations, classical large sample theory assumes that the data dimension $p$ is fixed, and therefore many  conclusions are not directly applicable to high-dimensional data when  $p$ increases with the sample size $N$. 
As analyzing high-dimensional data is of emerging interest in modern data science, it imposes new challenges to   understanding the statistical performance of the likelihood ratio test in the exploratory factor analysis, which will be investigated in the next section. 

\subsection{Main results}  \label{sec2.2}

In high-dimensional exploratory factor analysis,  it is important to understand the limiting behavior of the likelihood ratio test, as applying an inaccurate limiting distribution would lead to misleading scientific conclusions.
This section focuses on the limiting distribution of the likelihood ratio test under the null hypothesis, and investigates the influence of the data dimension $p$ and the sample size $N$ on the chi-square approximation.

Recent statistical literature has shown that the chi-square approximation for the likelihood ratio test  can  become inaccurate in various testing problems \citep{bai2009,jiang2013,He2018}, while this inaccuracy  issue   is still less studied in the exploratory factor analysis.  
To demonstrate that similar phenomena exist for the exploratory factor analysis, we first present a numerical example, before showing our theoretical  results. 

\begin{numexample}\label{eg:num1}
\textit{Consider $H_{0,0}$ in Section \ref{sec2.1}  with $N=1000$ and $p\in \{20, 100, 300, 500\}$. Under each combination of $(N,p)$, we  generate $X_i, i=1,\ldots, N$ from $\mathcal{N}(\mathbf{0}_p,\mathrm{I}_p)$ independently, and then compute the likelihood ratio test statistics $T_0$  in \eqref{eq1} 
and its Bartlett corrected version  $\rho_0 T_0$ in \eqref{eq2}. We repeat the procedure 5000 times, and present the histograms of $T_0$ and $\rho_0T_0$ in the first and second rows, respectively,  of  Figure \ref{fig:qqplot1}. For comparison, in each histogram, we add  the theoretical density curve of the limiting distribution $\chi^2_{f_0}$ in \eqref{eq1} and \eqref{eq2} (the red curves in Figure \ref{fig:qqplot1}).}
\end{numexample}

\begin{figure}[!ht]
\centering

\begin{subfigure}{\textwidth}
\centering
\includegraphics[width=0.235\textwidth]{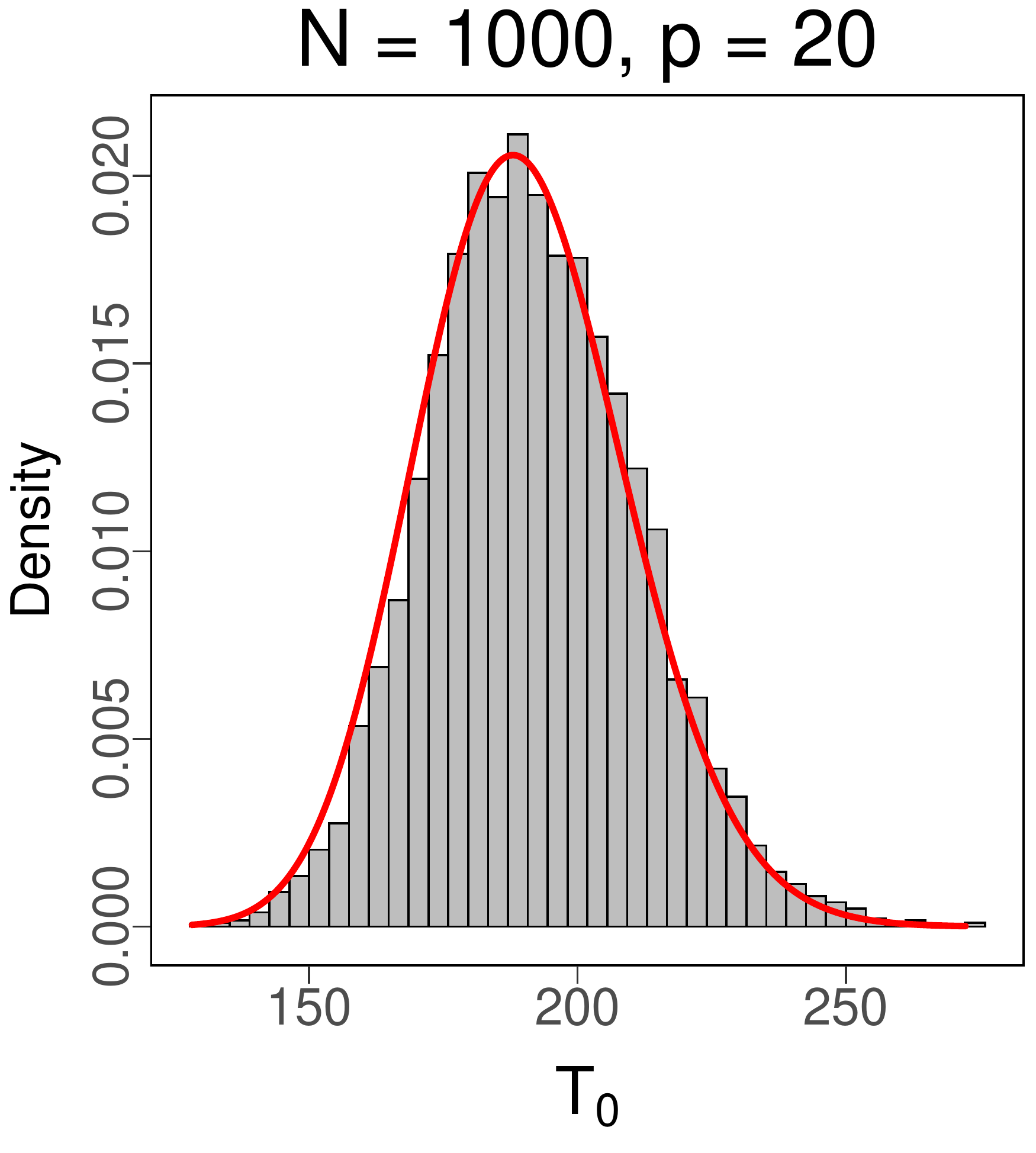}\ \ 
\includegraphics[width=0.235\textwidth]{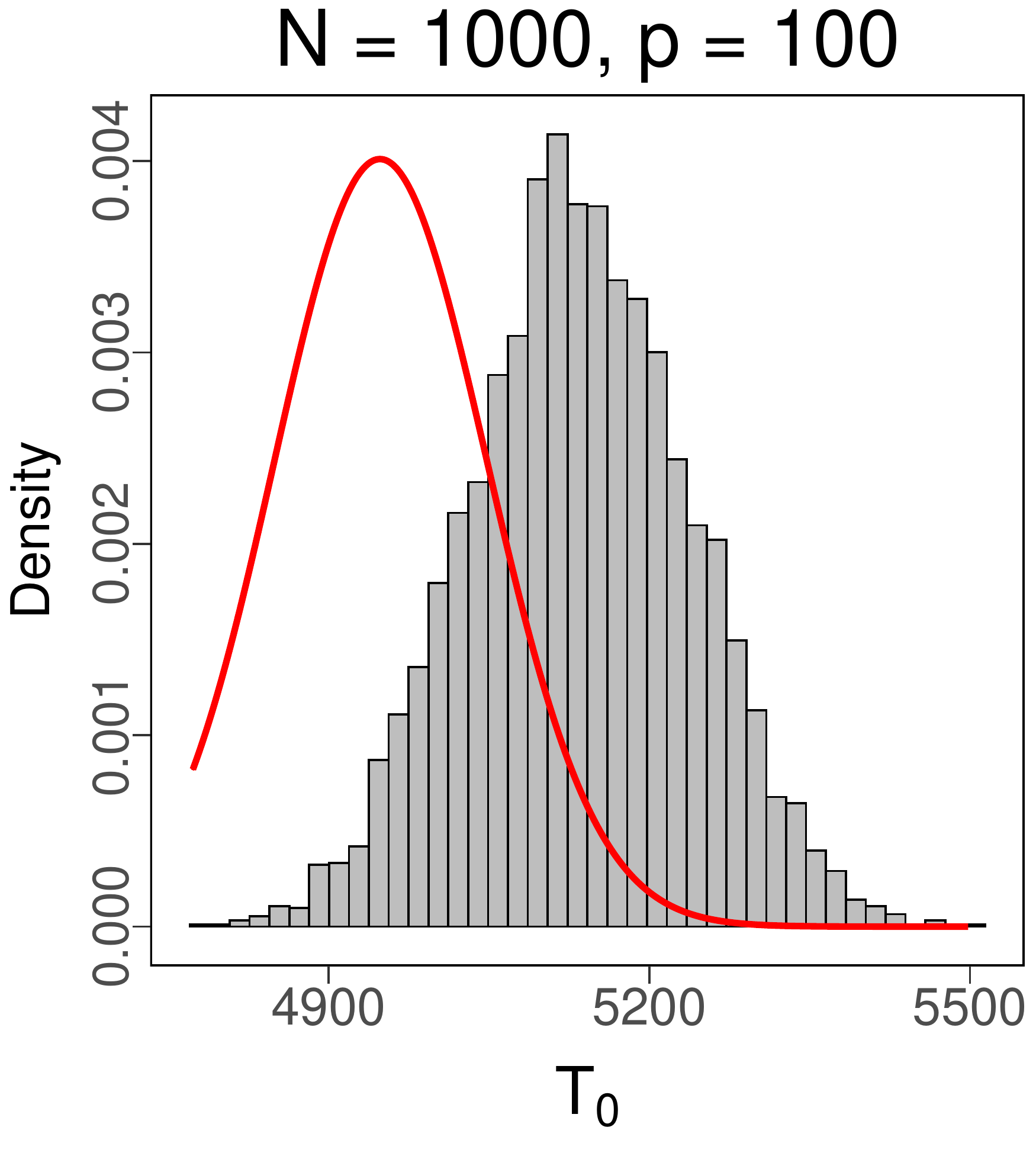}\ \ 
\includegraphics[width=0.235\textwidth]{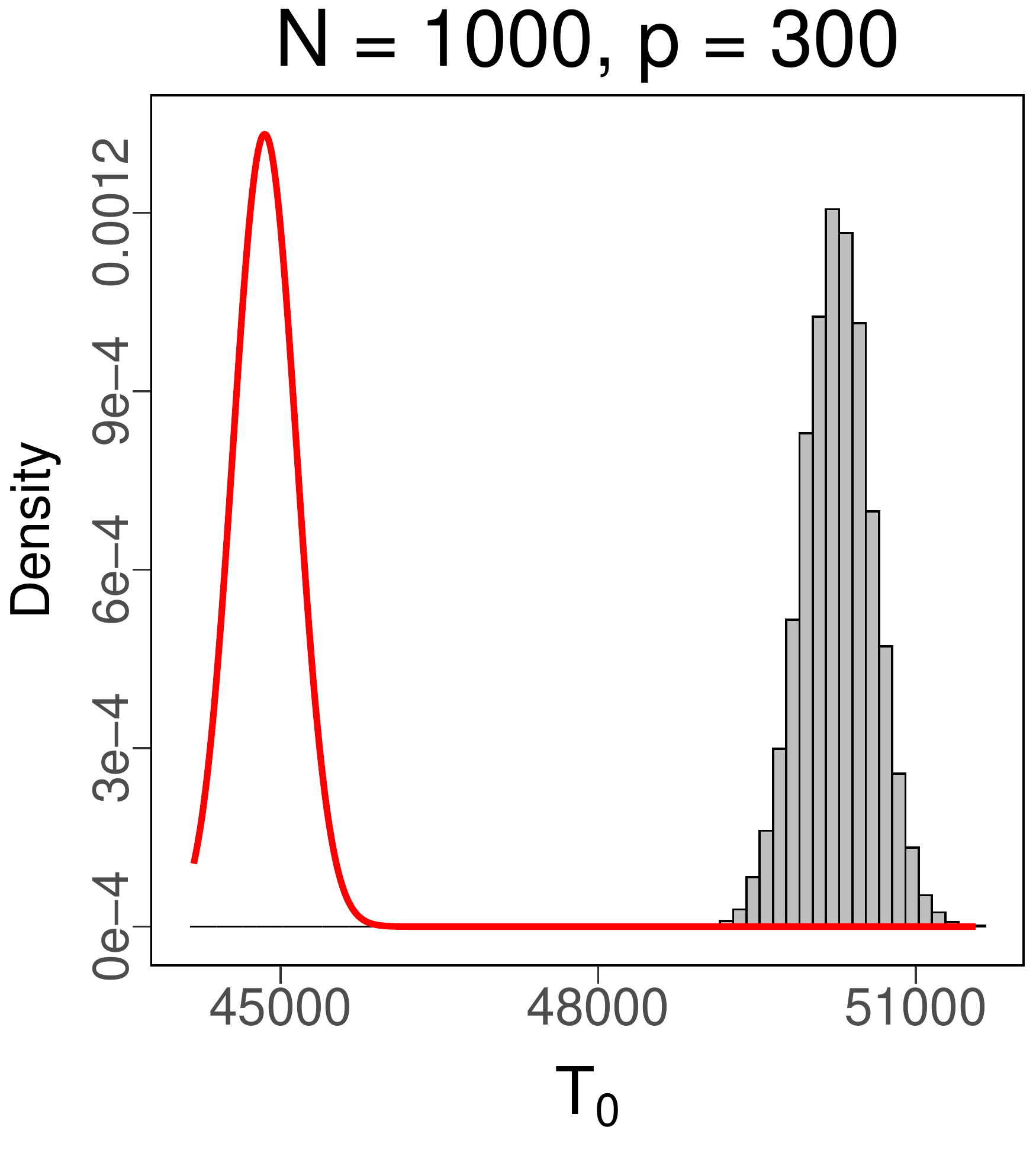}\ \ 
\includegraphics[width=0.235\textwidth]{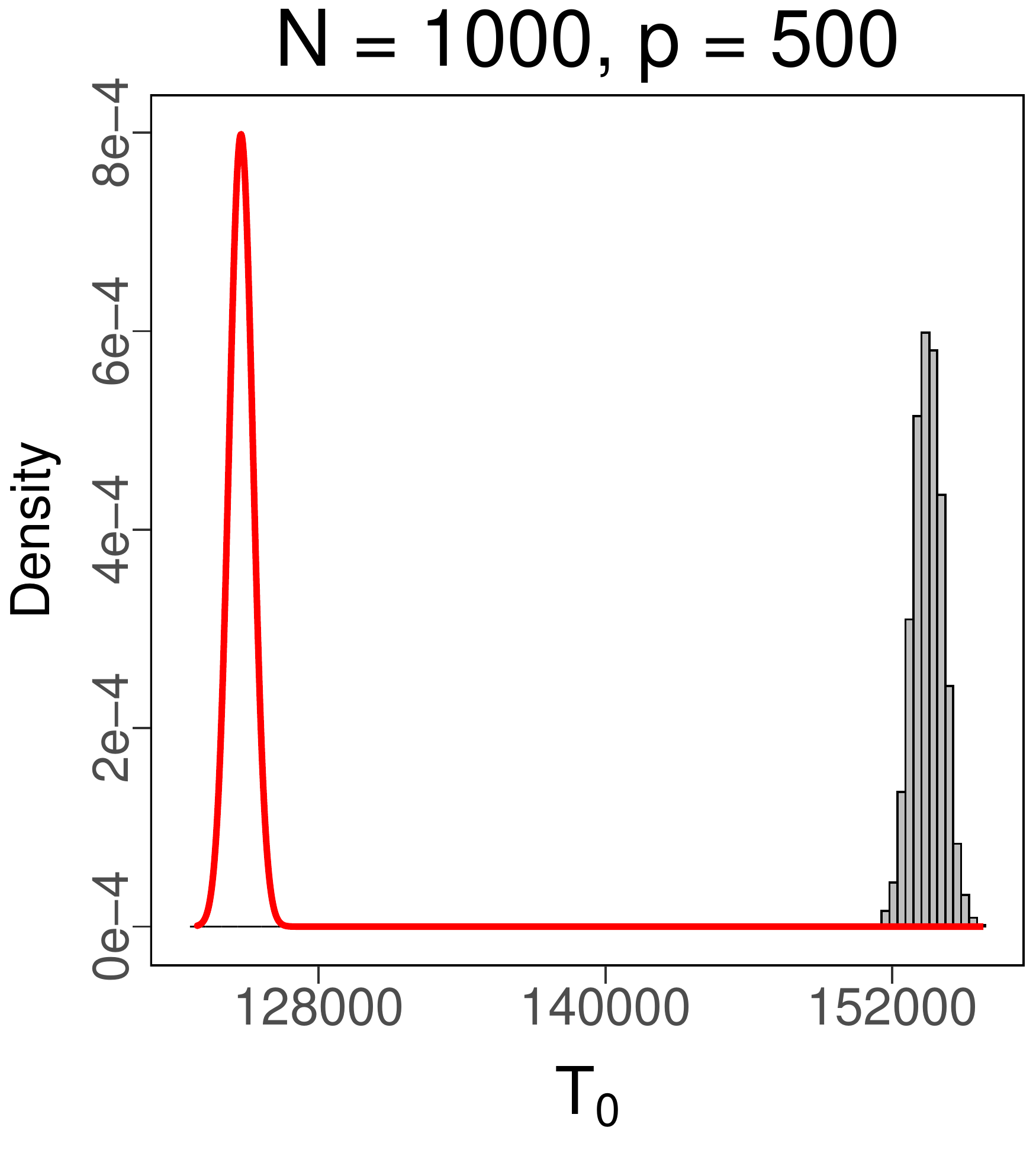}
\end{subfigure}\\[4pt]

\begin{subfigure}{\textwidth}
\centering
\includegraphics[width=0.235\textwidth]{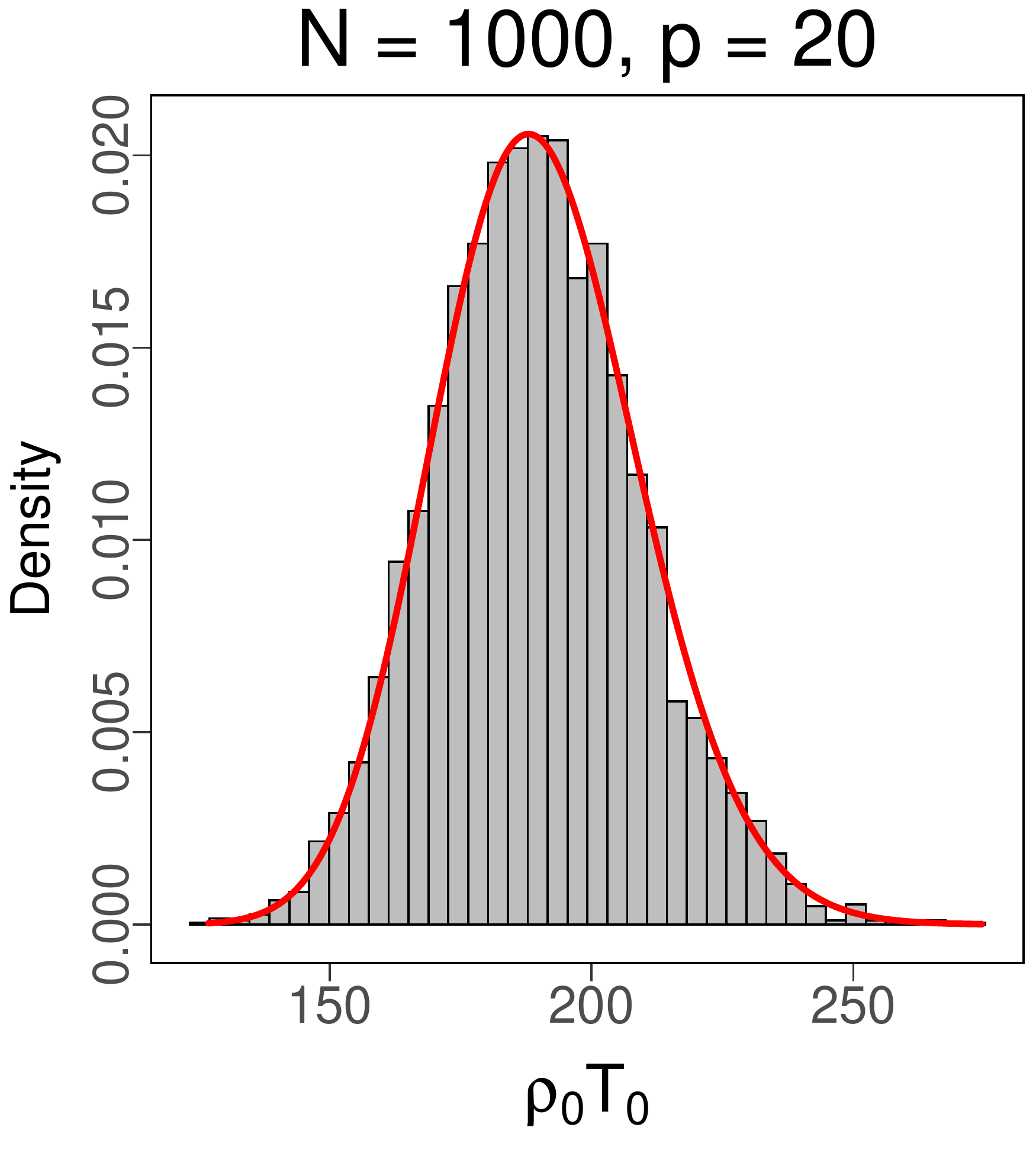}\ \  
\includegraphics[width=0.235\textwidth]{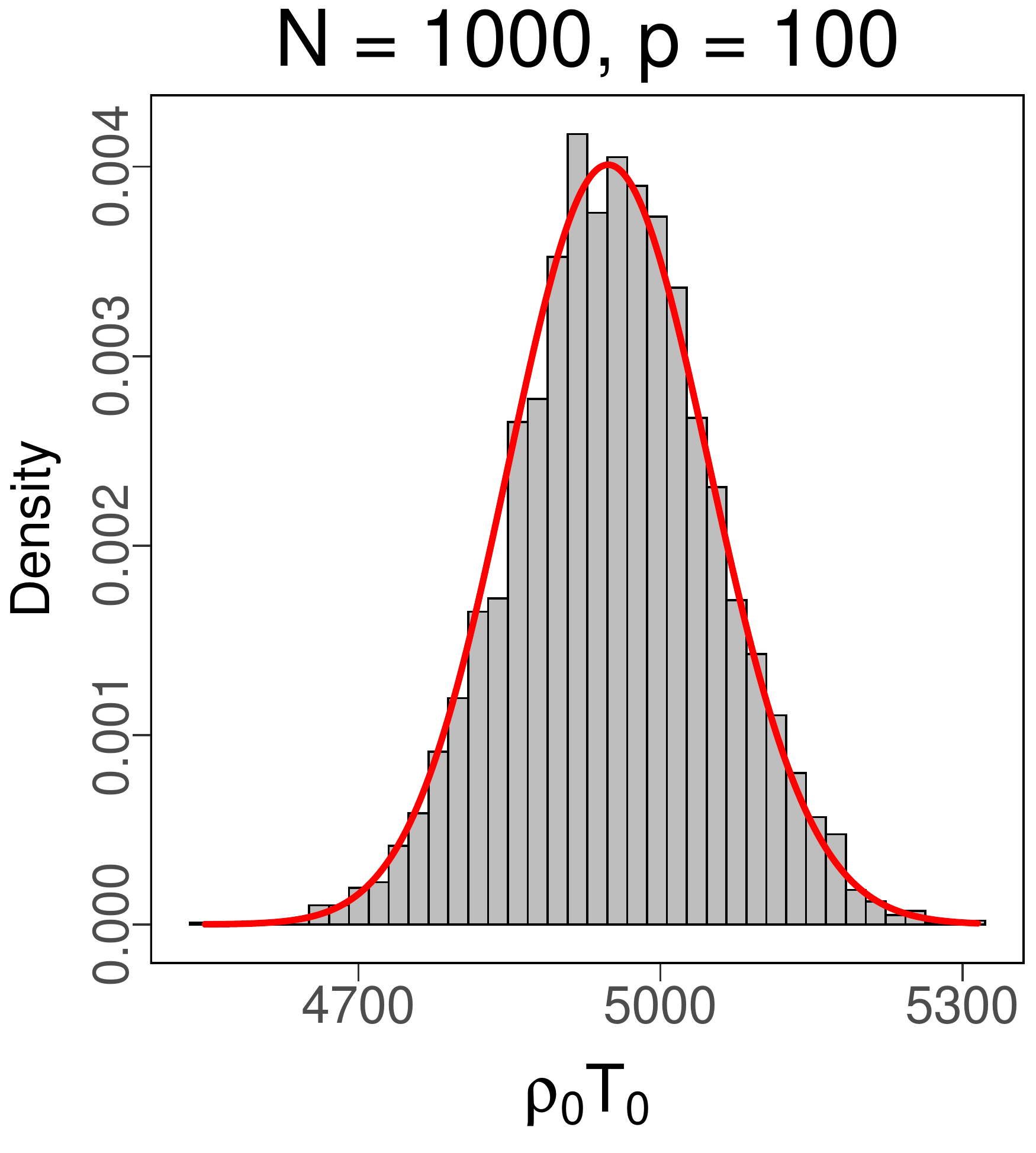}\ \ 
\includegraphics[width=0.235\textwidth]{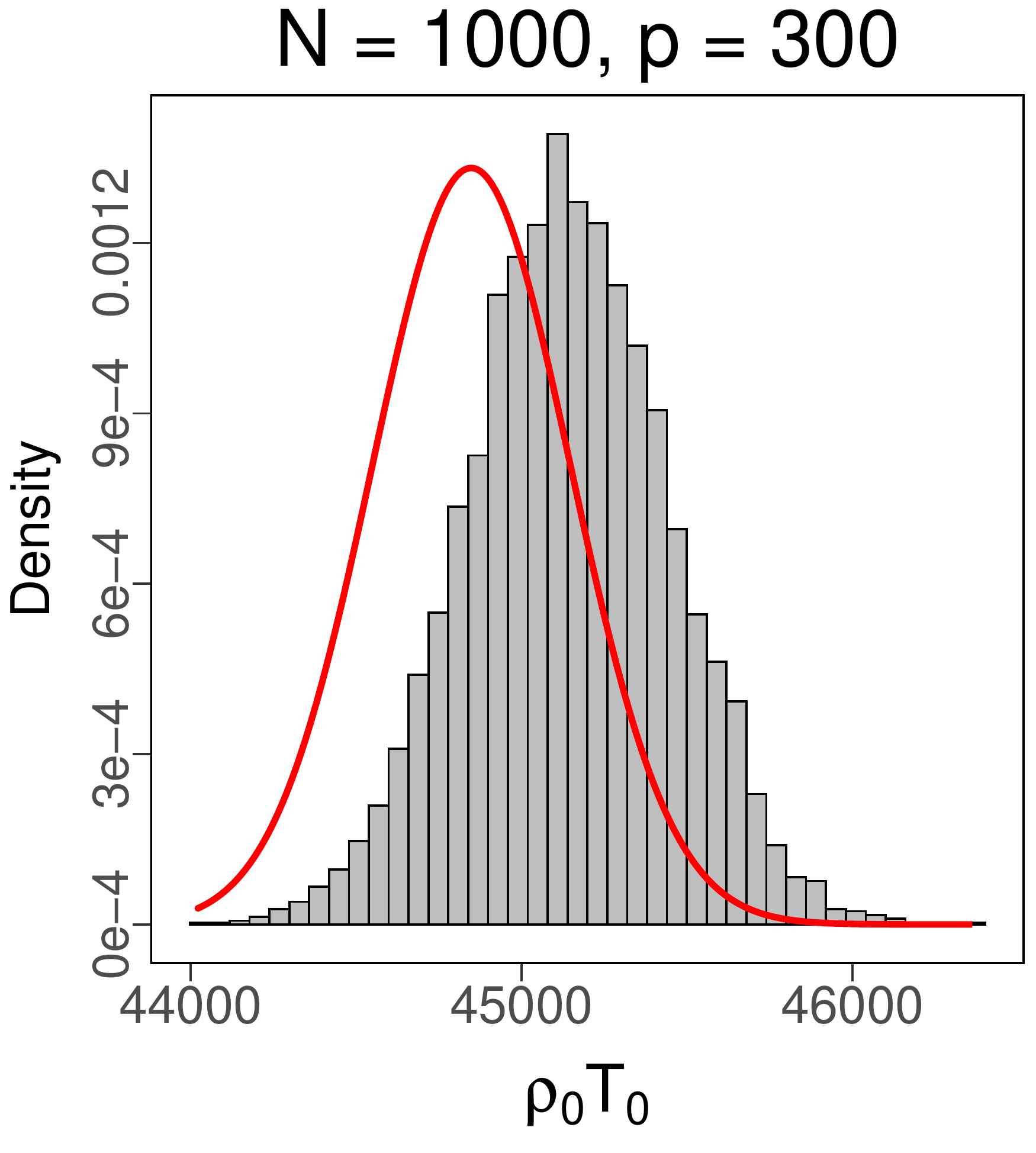}\ \ 
\includegraphics[width=0.235\textwidth]{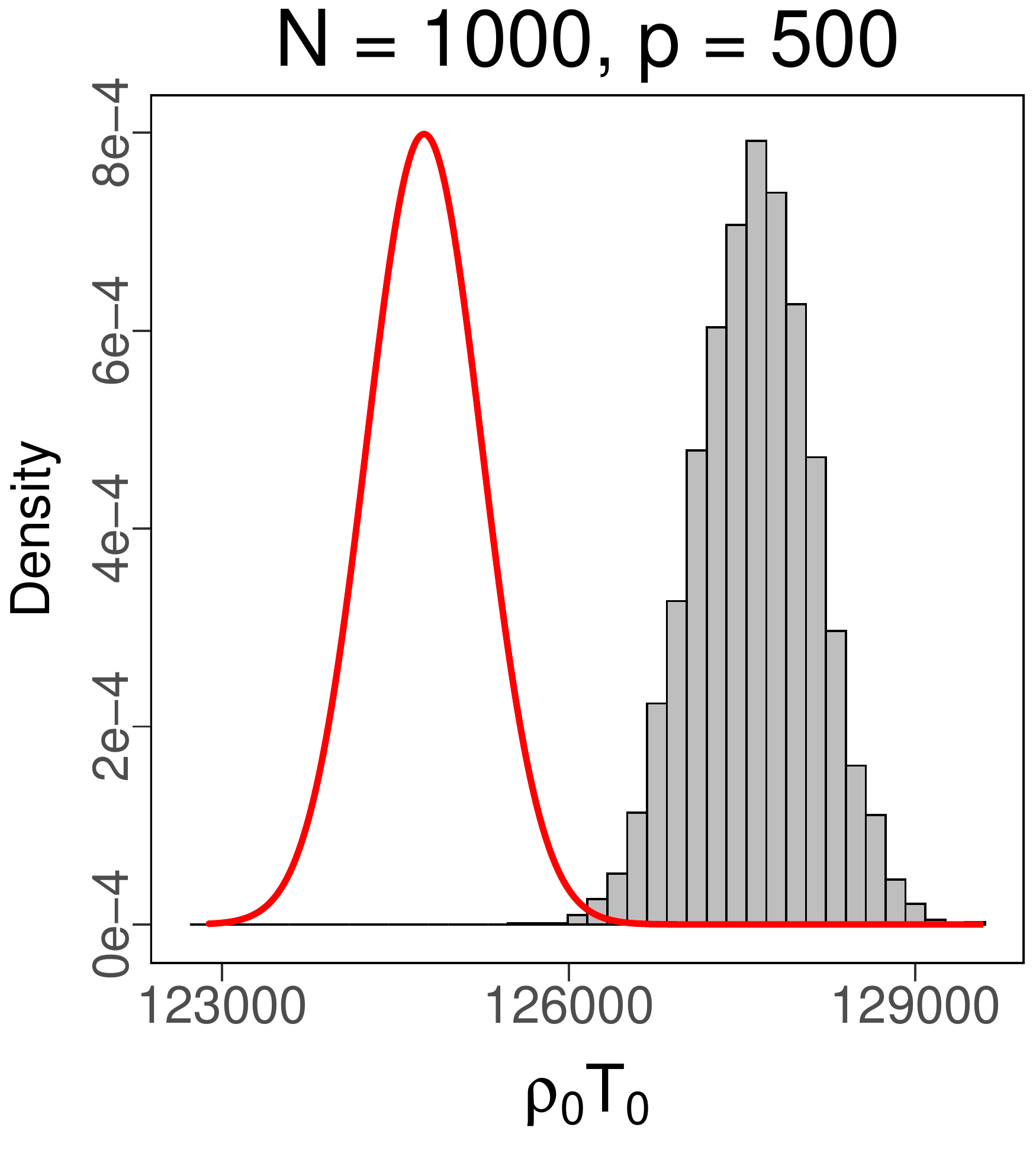}
\end{subfigure}


\caption{
Histograms of $T_0$ and $\rho_0T_0$ with the density curves of $\chi^2_{f_0}$} \label{fig:qqplot1}	
\end{figure}



From the two figures in the first column of  Figure \ref{fig:qqplot1}, 
we can see that when $p$ is small ($p=20$) compared to $N$, the density curve of $\chi^2_{f_0}$  approximates the histograms of $T_0$ and $\rho_0T_0$ well. This is consistent with the classical large sample theory in \eqref{eq1} and \eqref{eq2}.  However, as $p$ increases from 20 to 500, the density curve of $\chi^2_{f_0}$ moves farther away from the sample histograms of $T_0$ and $\rho_0T_0$, indicating the failure of the  chi-square approximation as $p$ increases. It is also interesting to note that the likelihood ratio test statistics without and with the Bartlett correction behave differently as $p$ increases, despite their similarity when $p$ is small. For instance, when $p=100$, $\chi_{f_0}^2$ already fails to approximate the distribution of $T_0$, but it can still well approximate that of the corrected statistic $\rho_0T_0.$  Nevertheless, when $p=300$ and 500, $\chi_{f_0}^2$ fails to approximate the distributions of both $T_0$ and $\rho_0T_0$, while the approximation biases differ. 
These numerical observations bring the following question in practice: how large the dimension $p$ with respect to the sample size $N$ can be so that we can still apply the classic chi-square approximation for the likelihood ratio test?




To provide a statistical insight into  this important practical issue, 
we derive the  necessary and sufficient condition to ensure the validity of the  chi-square approximation for the 
likelihood ratio test, as $p$ increases with $N$.  
Particularly, we first consider $H_{0,0}: k_0=0$ in Section \ref{sec2.1}, and provide the following Theorem \ref{thm1}.
 \begin{theorem} \label{thm1} 
 Suppose $N \geq p+5$. Let   $\chi_{f_0}^{2}(\alpha)$ denote the upper-level $\alpha$-quantile of the $\chi^2_{f_0}$ distribution. Under $H_{0,0}: k_0=0$, as $N\to \infty$, 
 \begin{enumerate}
   \item[(i)] $\sup_{\alpha \in (0,1)}|\Pr\{T_{0} > \chi_{f_0}^{2}(\alpha)\} - \alpha| \to 0,$ if and only if $\lim_{n\to\infty}p/N^{1/2} = 0$; 
\item[(ii)] $\sup_{\alpha \in (0,1)}|\Pr\{ \rho_0 \times T_{0} > \chi_{f_0}^{2}(\alpha)\} - \alpha| \to 0,$ if and only if $\lim_{n\to\infty}p/N^{2/3} = 0$. 
\end{enumerate}
\end{theorem}

%
 
In Theorem \ref{thm1}, $N \geq p+5$ is required for the technical proof. This condition is mild as $N \geq p+1$ is required for the existence of the likelihood ratio test statistic with probability one \citep{jiang2013}. 
Theorem \ref{thm1} (i) suggests that the chi-square approximation for $T_0$ in \eqref{eq1} starts to fail when the dimension $p$ approaches $N^{1/2}$, and (ii) shows that the chi-square approximation for $\rho_0T_0$ in \eqref{eq2} starts to fail when $p$ approaches $N^{2/3}$.  To further demonstrate the validity of Theorem \ref{thm1}, we conduct a simulation study as follows. 
 

\begin{numexample}\label{numex:2}
\textit{We take $p = \lfloor{N^{\varepsilon}}\rfloor$, where $N \in \{100, 500, 1000, 2000\}$ and $\varepsilon \in \{3/24,4/24,\ldots, 23/24\}$. For each combination of $(N,p)$, we generate  $X_i$ from $\mathcal{N}(\mathbf{0}_p,\mathrm{I}_p)$ for $i=1,\ldots, N$ independently, and conduct the likelihood ratio test with 
two chi-square approximations in \eqref{eq1}  and \eqref{eq2}, respectively. We repeat the procedure 1000 times to estimate the type \RNum{1} error rates with significance level 0.05, and then plot estimated type \RNum{1} error rates versus $\varepsilon$ in Figure \ref{fig:typeierror1}. 
 The left figure in Figure  \ref{fig:typeierror1} presents the results of the chi-square approximation for $T_0$ in \eqref{eq1}, where  the estimated type \RNum{1} error  begins to inflate when $\varepsilon$ approaches $1/2$. 
 In addition, the right figure in Figure  \ref{fig:typeierror1} presents the results of the chi-square approximation for $\rho_0T_0$ in \eqref{eq2}, where the estimated type \RNum{1} error  begins to inflate when $\varepsilon$ approaches $2/3$. 
 The two theoretical boundaries on $\varepsilon$ in Theorem \ref{thm1} are denoted by two vertical dashed lines in Figure  \ref{fig:typeierror1}.
 For each approximation, the theoretical and empirical values of $\varepsilon$ where the  approximation begins to fail are consistent.}
\end{numexample}

\bigskip
\begin{figure}[!ht]
\captionsetup[subfigure]{labelformat=empty}
\centering
\begin{subfigure}[b]{0.47\textwidth}
\centering
\caption{\quad \quad \small{Approximation \eqref{eq1} for $T_0$}} 
\includegraphics[width=\textwidth]{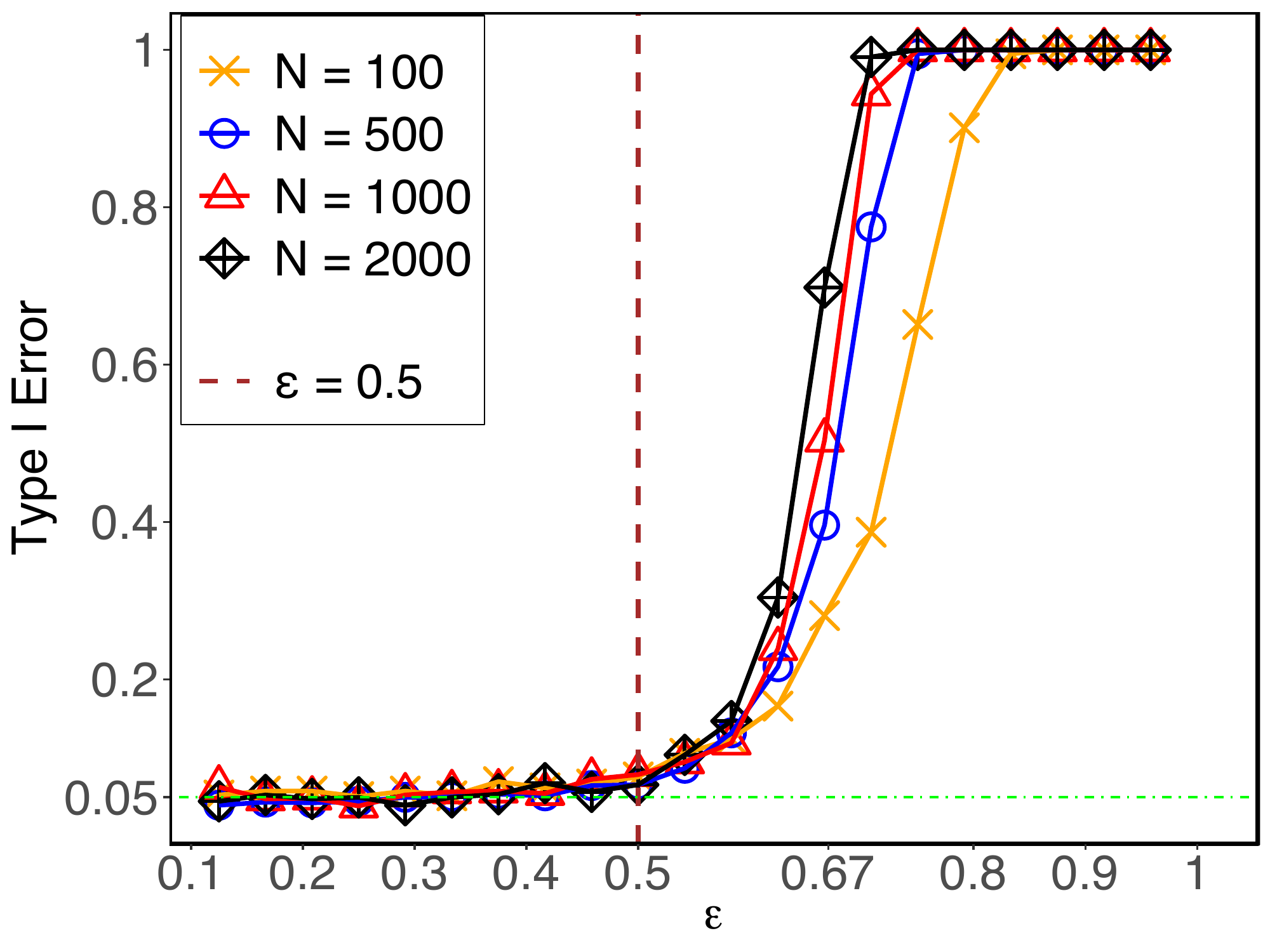}	 
\end{subfigure}\quad \quad 
\begin{subfigure}[b]{0.47\textwidth}
\centering
\caption{\quad \quad \small{Approximation \eqref{eq2} for $\rho_0T_0$}}
\includegraphics[width=\textwidth]{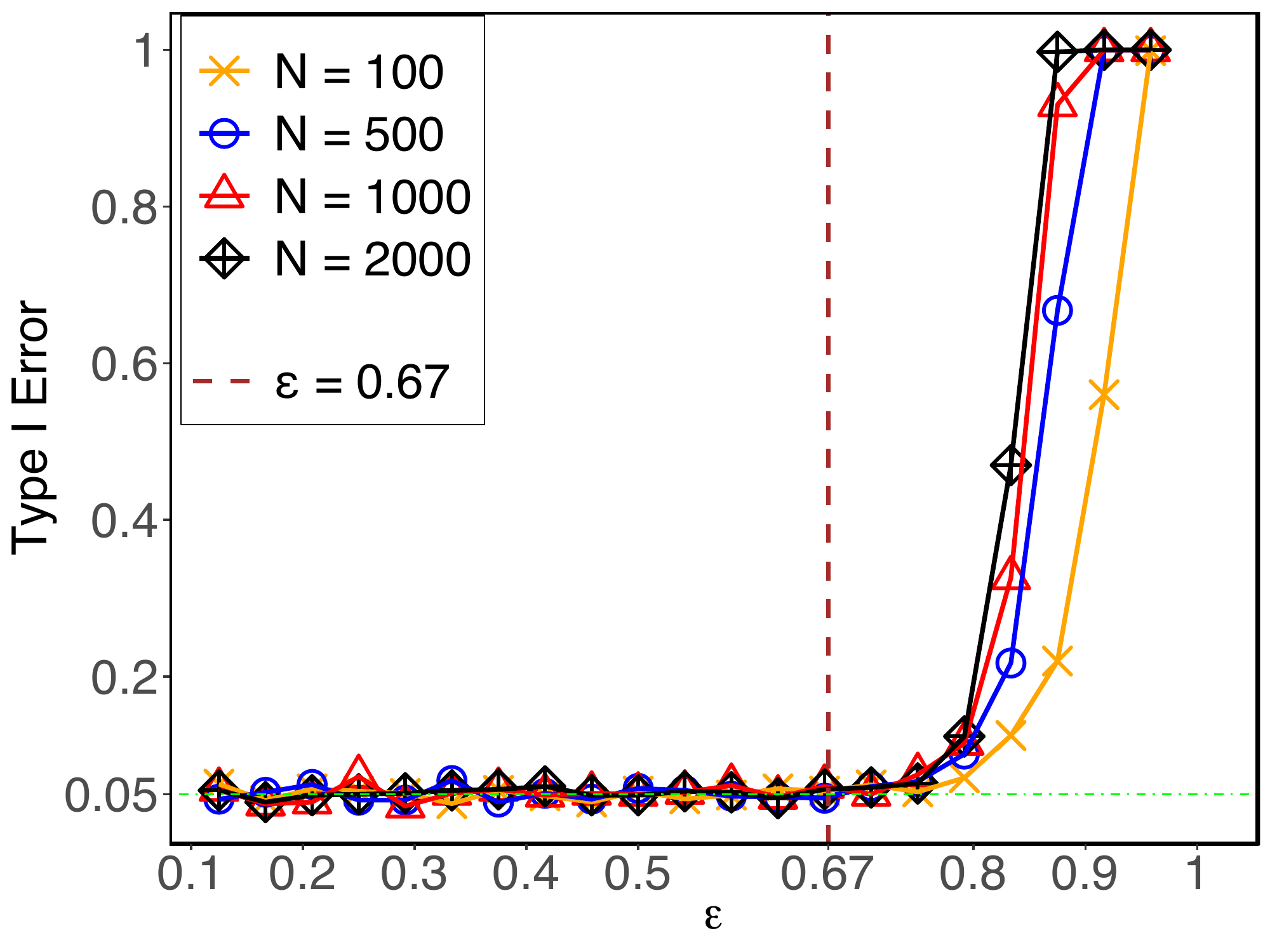}	
\end{subfigure}
\caption{Estimated type \RNum{1} error versus $\varepsilon$ when $k_0=0$}\label{fig:typeierror1}	
\end{figure}

We next investigate the sequential test for $H_{0,k}$ when $k\geq 1$. Under $H_{0,k}$, assume the true factor number is $k$, and  $\Lambda_k\Lambda_k^{\top}$ and $\Psi_k$ are the true values such that \eqref{eq01} holds with $\Lambda \Lambda^{\top}=\Lambda_k\Lambda_k^{\top}$ and $\Psi=\Psi_k$, where $\Lambda_k$ is a matrix of size $p\times k$, and $\Psi_k$ is a diagonal matrix. In classical multivariate analysis with fixed dimension and certain regularity conditions, it can be shown that  $\hat{\Lambda}_k\hat{\Lambda}_k^{\top}\xrightarrow{P} \Lambda_k\Lambda_k^{\top}$ and $\hat{\Psi}_k\xrightarrow{P} \Psi_k$, where $\xrightarrow{P}$ represents the convergence in probability; see, e.g., Theorem 14.3.1 in \cite{anderson1958introduction}. 
To facilitate the following theoretical analysis, we consider a simplified version of the test by assuming  $\Lambda_k \Lambda_k^{\top}$ and $\Psi_k$ are given, and define $\Sigma_k=\Lambda_k \Lambda_k^{\top}+\Psi_k$. 
Then we consider testing $H_{0,k}': \Sigma = \Sigma_k$, and the 
likelihood ratio test statistic can be expressed as 
$$T' = - (N-1)\log(|\hat{\Sigma}| \times |\Sigma_k|^{-1})+ (N-1)\{\mathrm{tr}(\hat{\Sigma} \Sigma_k^{-1}) - p\};$$
see Section 8.4 of  \cite{muirhead2009aspects}.
The test statistic $T'$  and $T_k$ in \eqref{eq:tkstat} are the same except that $T'$ is based on the true value $\Sigma_k=\Lambda_k \Lambda_k^{\top}+\Psi_k$, while $T_k$ is based on $\hat{\Sigma}_k=\hat{\Lambda}_k \hat{\Lambda}_k^{\top}+\hat{\Psi}_k$, with $\hat{\Lambda}_k \hat{\Lambda}_k^{\top}$ and $\hat{\Psi}_k$ being the maximum likelihood estimators of $\Lambda_k \Lambda_k^{\top}$ and $\Psi_k$, respectively, under the $k$-factor model.
Under the classical setting with $p$ fixed, the chi-square approximation of $T'$ is $T'\xrightarrow{D}\chi^2_{f'},$ where $f'=p(p+1)/2$, and 
by the Bartlett correction with $\rho'=1- \{6(N-1)(p+1)\}^{-1} (2p^2+3p-1)$, we have  
 $\rho' T'\xrightarrow{D}\chi^2_{f'}$. For this simplified testing problem $H_{0,k}'$, the test statistic $T'$ and its limit do not depend on the number of factors $k$, as the true 
 $\Lambda_k\Lambda_k^{\top}$ and $\Psi_k$ are assumed to be given. 


Considering $H_{0,k}'$ and the statistic $T'$, we next provide the necessary and sufficient condition on when the chi-square approximation for the likelihood ratio test fails as the data dimension $p$ increases under $H_{0,k}'$. 
\begin{theorem} \label{thm2} 
 Suppose $N\geq p+2$. Under $H_{0,k}':\ \Sigma = \Lambda_k\Lambda_k^{T} +\Psi_k$, with given $\Lambda_k$ and $\Psi_k$, and $k=k_0$, as $N\to \infty$, 
 \begin{enumerate}
   \item[(i)] $\sup_{\alpha \in (0,1)}|\Pr\{T' > \chi_{f'}^{2}(\alpha)\} - \alpha| \to 0,$ if and only if $\lim_{n\to\infty}p/N^{1/2} = 0$; 
\item[(ii)] $\sup_{\alpha \in (0,1)}|\Pr\{ \rho' \times T' > \chi_{f'}^{2}(\alpha)\} - \alpha| \to 0,$ if and only if $\lim_{n\to\infty}p/N^{2/3} = 0$. 
\end{enumerate}
\end{theorem}

\begin{remark}\label{rm:divergep}
For the more general testing problem $H_{0,k}$,  we need to obtain the maximum likelihood estimators $\hat{\Lambda}_k$ and $\hat{\Psi}_k$, and then conduct the likelihood ratio test with chi-square  approximations \eqref{eq3} or \eqref{eq4}. 
When the number of latent factors $k$ is fixed compared to $N$ and $p$, we note that  $\rho_k/\rho'$ and $f_k/f'$ asymptotically converge to 1.  
Furthermore, if  $\hat{\Lambda}_k\hat{\Lambda}_k^{\top}+\hat{\Psi}_k$  approximates the true $\Lambda_k\Lambda_k^{\top}+\Psi_k$ sufficiently well, we  expect that the  conclusions in Theorem \ref{thm2} would hold for the likelihood ratio test under the null hypothesis $H_{0,k}$ similarly. 
In particular, when $k$ is fixed as $N \to \infty$, consistent estimation of ${\Lambda}_k$ and ${\Psi}_k$ has been discussed under  both fixed $p$ 
 in the classical literature   
\cite[see, e.g.,][Theorem 14.3.1]{anderson1958introduction} and $p\to \infty$ 
 in recent literature on high-dimensional factor analysis model 
\citep[see, e.g.,][]{bai2012statistical}. 
When $k$ also diverges with $N$ and $p$, an asymptotic regime that is less investigated in the literature,
 deriving a similar condition for the chi-squared approximation would require accurate characterizations of the biases of estimating $\Lambda_k$ and $\Psi_k$,  
which, however, would be  challenging and need new developments of high-dimensional theory and methodology. 
\end{remark}

 We next demonstrate the theoretical results through the following numerical study.



\begin{numexample}\label{numex:3}
\textit{We consider the likelihood ratio test under $H_{0,k}$ with $k=k_0 \in\{1,3\}$. 
(I) When $k_0=1,$ under $H_{0,1}$, we set 	$\Lambda = \rho \times \mathbf{1}_{p}$ and $\Psi = (1-\rho^{2})\mathrm{I}_{p}$, with $\rho = 0.3$.
(II) When $k_0=3$,  under $H_{0,3}$, we  set $\Psi = (1-\rho^{2})\mathrm{I}_{p}$ and 
$$\Lambda = 
    \begin{bmatrix}
    \rho \times \mathbf{1}_{p_1}& \mathbf{0}_{p_1}& \mathbf{0}_{p_1}\\
    \mathbf{0}_{p_1}&\rho \times \mathbf{1}_{p_1}& \mathbf{0}_{p_1}\\
   \mathbf{0}_{p-2p_1}& \mathbf{0}_{p-2p_1}&\rho \times \mathbf{1}_{p-2p_1}
    \end{bmatrix},$$
where $p_1=\lfloor p/3 \rfloor$, $\rho = 0.6$, and $\mathbf{1}_{p_1}$ represents a $p_1$-dimensional vector with all one entries.  
For both cases, we  set $p = \lfloor{N^{\varepsilon}}\rfloor$, where $N \in \{100, 500,1000,2000\}$ and $\varepsilon \in \{8/24,7/24,...,23/24\}$.
And we generate each observation $X_i, i=1,\ldots, N$, from $\mathcal{N}(\mathbf{0},\Lambda\Lambda^{\top}+\Psi)$ independently, and conduct the likelihood ratio test with the function \textsf{factanal()} in \textsf{R}. Similarly to Figure \ref{fig:typeierror1}, we plot the estimated type \RNum{1} error rates (based on 1000 replications) versus $\varepsilon$  for two approximations \eqref{eq3} and \eqref{eq4}, where the results of case (I) are in Figure \ref{fig:typeierror2}, and the results of case (II) are in Figure \ref{fig:typeierror3}.}
\end{numexample}


%

\bigskip
\begin{figure}[!ht]
\captionsetup[subfigure]{labelformat=empty}
\centering
\begin{subfigure}[b]{0.47\textwidth}
\centering
\caption{\quad \quad \small{Approximation \eqref{eq3} for $T_1$}}
\includegraphics[width=\textwidth]{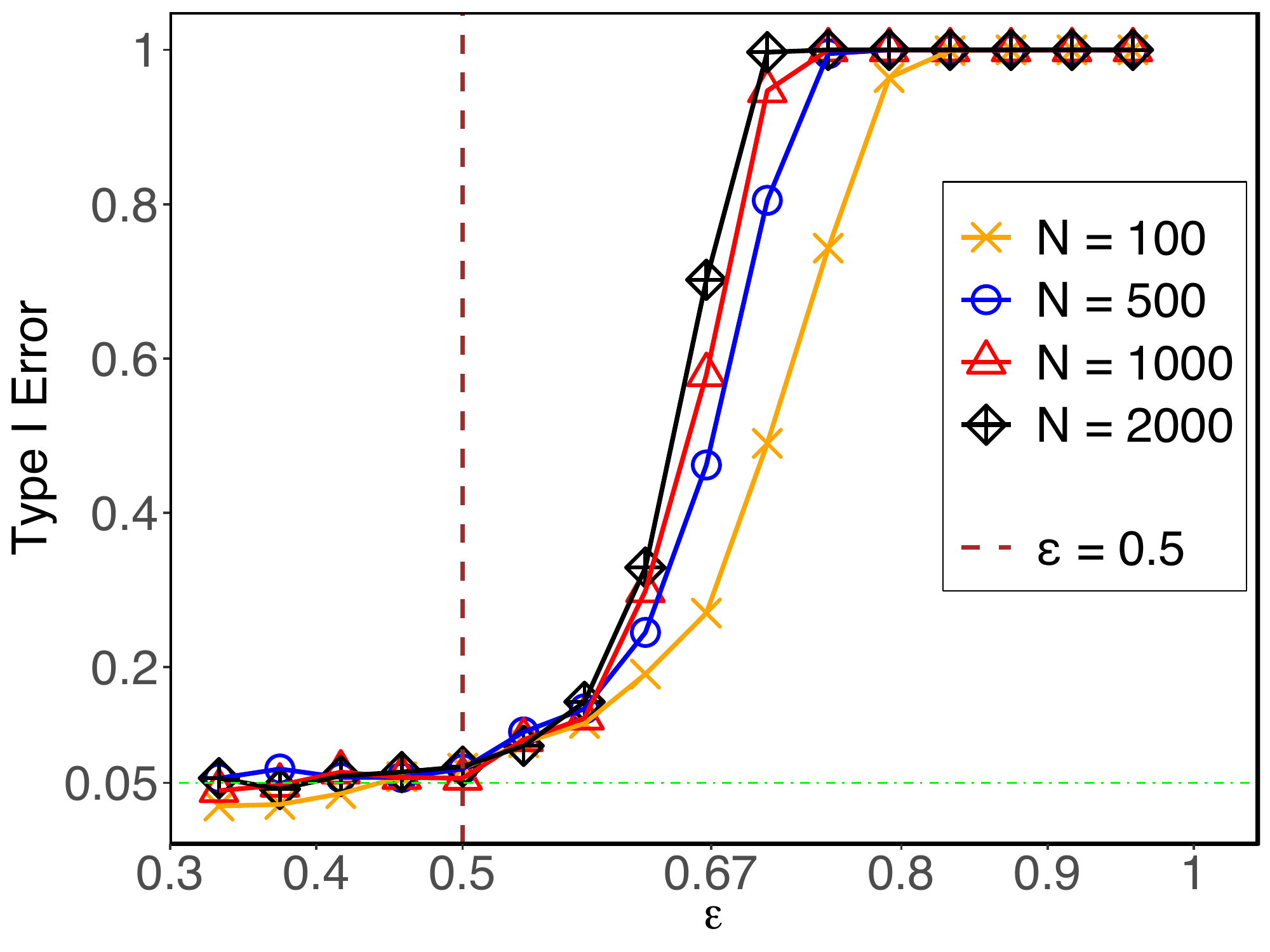}	
\end{subfigure}\quad \quad 
\begin{subfigure}[b]{0.47\textwidth}
\centering
\caption{\quad \quad \small{Approximation \eqref{eq4} for $\rho_1T_1$}}
\includegraphics[width=\textwidth]{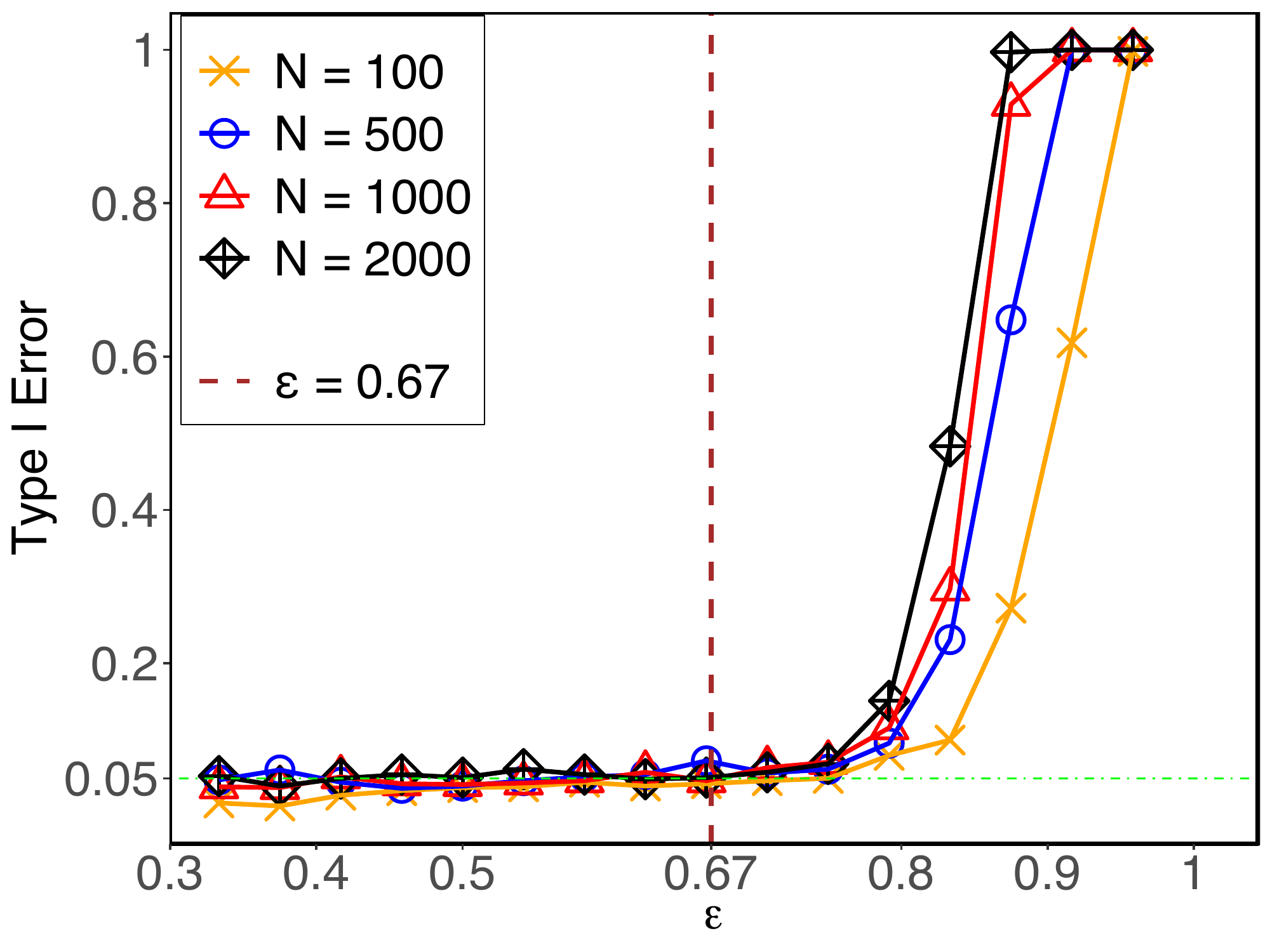}	
\end{subfigure}
\caption{Estimated type \RNum{1} error versus $\varepsilon$ when $k_0=1$}\label{fig:typeierror2}	
\end{figure}

%
%
%

\begin{figure}[!ht]
\captionsetup[subfigure]{labelformat=empty}
\centering
\begin{subfigure}[b]{0.47\textwidth}
\centering
\caption{\quad \quad \small{Approximation \eqref{eq3} for $T_3$}}
\includegraphics[width=\textwidth]{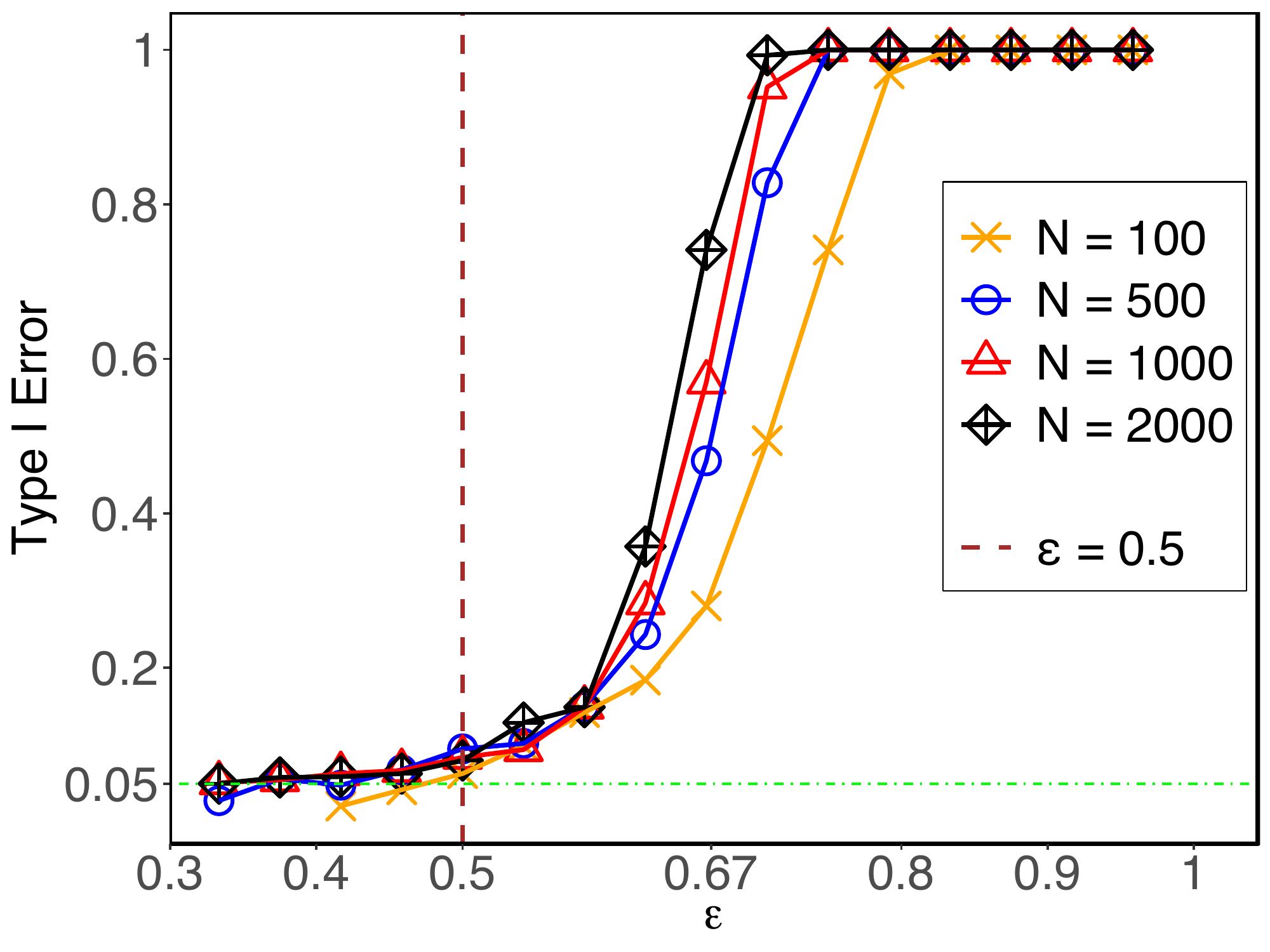}	
\end{subfigure}\quad \quad 
\begin{subfigure}[b]{0.47\textwidth}
\centering
\caption{\quad \quad \small{Approximation \eqref{eq4} for $\rho_3T_3$}}
\includegraphics[width=\textwidth]{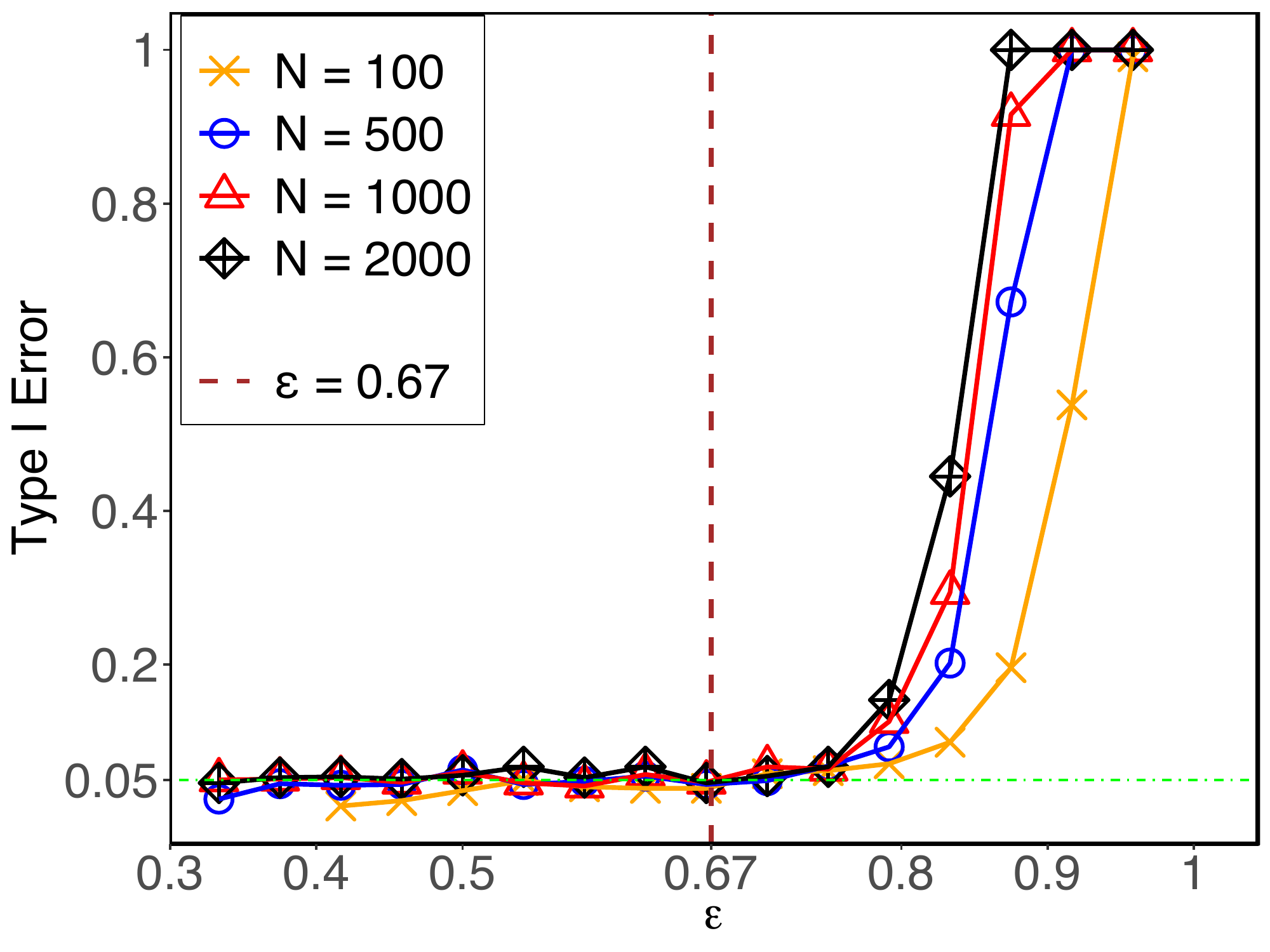}	
\end{subfigure}
\caption{Estimated type \RNum{1} error versus $\varepsilon$ when $k_0=3$}\label{fig:typeierror3}	
\end{figure}

Similarly to Numerical Example \ref{numex:2}, Numerical Example \ref{numex:3} also demonstrates that 
  the empirical values of $\varepsilon$, where the chi-square approximations start to fail, are consistent with the corresponding theoretical results. The necessary and sufficient conditions therefore would provide simple quantitative guidelines to check in practice. 
 In addition, it is worth mentioning that the conditions in Theorems \ref{thm1} and \ref{thm2} also reflect the biases of the chi-square approximations. For instance, considering the likelihood ratio test for $H_{0,0}$,  by the proof of Theorem \ref{thm1},  when $p/N \to 0$, we obtain  that $\mathrm{E}(T_0-\chi^2_{f_0})\times \{\mathrm{var}( \chi^2_{f_0} )\}^{-1/2} $ 
 is approximately $C_1 p^2/N,$ 
 and   $\mathrm{E}(\rho_0\times T_0-\chi^2_{f_0})\times \{\mathrm{var}( \chi^2_{f_0} )\}^{-1/2}$ is approximately  $C_2 p^3/N^2,$ where $C_1$ and  $C_2$ are positive constants.  This suggests that the mean of the chi-square limit will become  smaller than the means of $T_0$ and $\rho_0T_0$ as $p$ increases, which is consistent with the observed phenomenon in Figure \ref{fig:qqplot1}.

Moreover, Figures  \ref{fig:typeierror1}--\ref{fig:typeierror3} show   
  that the estimated type \RNum{1} error of the likelihood ratio test increases as $\epsilon$ increases. 
This can provide one possible explanation for the  well-known finding that the likelihood ratio test tends to overestimate the number of factors \citep{Hayashi2007}. 
In particular, let $\hat{k}$ denote the number of factors estimated by the sequential procedure described in Section \ref{sec2.1}, and let $k_0$ denote the  true number of factors. 
Note that in the sequential procedure, rejecting $H_{0,k_0}$ leads to an overestimation of the number of factors, i.e.,  $\hat{k}>k_0$.
Thus, when the type \RNum{1} error of testing $H_{0,k_0}$ inflates  as in Figures  \ref{fig:typeierror1}--\ref{fig:typeierror3}, 
the probability of rejecting $H_{0,k_0}$ would also increase, which consequently suggests an inflation of the probability of overestimating the number of factors, $\hat{k}>k_0$.
We also conduct simulation studies in Section \ref{sec:simfactor} to demonstrate the performance of estimating the number of factors using the likelihood ratio test.
The numerical results are consistent with the above theoretical analyses and show that the procedure begins to overestimate the number of factors when the type \RNum{1} error begins to inflate.
 

Furthermore, Theorems \ref{thm1} and \ref{thm2} indicate that given the same sample size, the chi-square approximation with the Bartlett correction can hold for a larger $p$ than the one without the Bartlett correction.  
 This explains the patterns observed in Figure \ref{fig:qqplot1}.    
Under the classical settings where $p$ is fixed, researchers have shown that the Bartlett correction can improve the convergence rate of the likelihood ratio test statistic from $O(N^{-1})$ to $O(N^{-2})$; however, this result does not apply to the high-dimensional setting  with $p$  increasing with $N$. Our theoretical results in Theorems \ref{thm1} and \ref{thm2} provide a more precise description on how the Bartlett correction improves the chi-square approximations for  high-dimensional data, in terms of the failing boundary of $p$ with respect to $N$.

\begin{remark} \label{rm:difference}
Similar phase transition phenomena were  discussed in \cite{he2020phase}. 
However, we point out that this paper considers different problem settings. 
In particular, \cite{he2020phase} discussed several problems on testing  mean vectors and covariances, whereas   
Theorem \ref{thm1} examines testing correlation matrices. 
Moreover, Theorem \ref{thm2} considers a problem of testing the covariance equal  to a given $k$-factor matrix, which was not discussed in  \cite{he2020phase}. To establish the result, it is required to derive a new high-dimensional asymptotic result given as Lemma \ref{lm:1} in the Appendix of this paper.
\end{remark}

\section{Discussions}\label{sec5}


This paper investigates the influence of the data dimension on the popularly used likelihood ratio test in high-dimensional exploratory factor analysis. 
For the  likelihood ratio test without or with the Bartlett correction, we derive the necessary and sufficient conditions to ensure the validity of the chi-square approximations under the corresponding null hypothesis.
The developed theoretical conditions only depend on the relationship between the data dimension and the sample size, and would provide simple quantitative guidelines to check in practice. 




The theoretical results in this paper are established under the common normality assumption of the observations $\{X_i, i=1,\ldots, N\}$.
To illustrate the robustness of the theoretical results to the normality assumption, we  conduct additional simulation studies with $X_i$'s  following a discrete distribution or a heavy-tailed $t$-distribution in  Appendix. Similar numerical findings are observed when detecting the existence of factors, which suggests that the validity of the theoretical results and the usefulness of the developed conditions in practice.  Please see Section \ref{sec:addsimu} in Appendix. 

Moreover, this paper focuses on controlling the type \RNum{1} error when testing a given null hypothesis,
	whereas 
	 deciding the number of factors  would involve  multiple steps of hypothesis testing in the sequential procedure.
 When the derived phase transition conditions are satisfied, our theoretical results suggest that the type I error of testing  corresponding null hypothesis can be asymptotically controlled. 	 
However, the probability of 
correctly deciding the true number of factors    
 relies on not only the type \RNum{1} error but also the power of testing each hypothesis in the sequential procedure. 
The power of the likelihood ratio test depends on certain complicated hypergeometric functions \citep{muirhead2009aspects}, which would be very challenging to investigate under high dimensions. 
We would like to leave this interesting problem as a  future study. 
In addition to the likelihood ratio test, it is also of interest to develop other efficient methods for deciding the number of factors in high-dimensional settings  \citep[see, e.g.,][]{bai2002determining,chen2020determining}. 
When applying the likelihood ratio test in the exploratory factor analysis, it is worth noting that the data dimension $p$ is not the only condition to consider.  
Researchers have discussed various other regularity conditions such as small sample size \citep{maccallum1999sample,mundfrom2005minimum,Winter2009,Winter2012}, nonnormality \citep{yuan2002unified,barendse2015using}, and  rank deficiency \citep{Hayashi2007}.  
 The results in this paper only provide one necessary requirement to check in the high-dimensional exploratory factor analysis. 

The results in this paper are also 
related to the important design problem on minimum sample size requirement for the exploratory factor analysis
 \citep{velicer1998affects,mundfrom2005minimum}. 
 The existing literature have conducted extensive simulation studies to explore what is the minimum sample size $N$ required or how large the ratio $N/p$ should be. In this paper, we derive theoretical results  suggesting that we may also consider the polynomial relationship between $N$ and $p$.  Specifically, given the number of variables $p$ to consider, the sample size should be at least $p^2$ to apply the likelihood ratio test, and at least $p^{3/2}$ to apply the likelihood ratio test with the Bartlett correction. This may provide helpful  statistical insights into the practice of exploratory factor analysis.

Although this paper focuses on the exploratory factor analysis, we expect that the failure of chi-square approximations under high dimensions can happen generally in other latent factor modeling problems such as the confirmatory factor analysis \citep{thompson2004exploratory,koranindicators2020} and the exploratory item factor analysis \citep{reckase2009multidimensional,chen2019joint}. Moreover, the phenomena introduced in this paper may also occur for other fit indexes that involve certain chi-square limit, such as the root mean square error of approximation  \citep{steiger2016notes}.
New high-dimensional theory and methodology for these problems would need to be further investigated.

\section*{Acknowledgement}
The authors are grateful to the Editor-in-Chief Professor Matthias von Davier, an Associate Editor, and  three referees for their valuable comments and suggestions. This research is partially supported by NSF CAREER SES-1846747, DMS-1712717, and SES-1659328.

\bibliographystyle{apalike}
\bibliography{Res_Bib}

\vspace{4em}


\begin{center}
	\bf{\Large{Appendix}}
\end{center}

\appendix

\noindent This appendix presents the technical proofs in Section \ref{sec:pfall} and additional simulations in Section \ref{sec:appendixsimu}. 




\section{Proofs} \label{sec:pfall}

We prove  Theorems \ref{thm1} and \ref{thm2} in Sections \ref{sec:pfthm1} and \ref{sec:pfthm2}, respectively, and provide a required lemma and its proof in Section \ref{sec:lemma}. 
In the following proofs, 
for two sequences of number $\{a_N: N\geq 1\}$ and $\{b_N: N\geq  1\}$, $a_N=O(b_N)$ denotes $\limsup_{n\to \infty}|a_N/b_N|<\infty$, and $a_N=o(b_N)$ denotes $\lim_{N\to \infty}a_N/b_N=0.$


\subsection{Proof of Theorem \ref{thm1}}\label{sec:pfthm1}
To derive the necessary and sufficient condition on the dimension of data,  it is required to correctly understand the limiting behavior of the likelihood ratio test statistic under both  low- and high-dimensional settings. In particular, we examine the limiting distribution of the  likelihood ratio test statistic based on its moment generating function.
For easy presentation in the technical proof, we let $n=N-1$ below. 
Then we can write $T_0=-n\log |\hat{R}_n|$.  Under the conditions
of Theorem \ref{thm1}, by Theorem 5.1.3 in \citet{muirhead2009aspects} and  Lemma 5.10 in \citet{jiang2013}, we know that there exists a small constant $\delta_0>0$ such that for $h\in (-\delta_0, \delta_0)$,  
\begin{eqnarray*}
\mathrm{E}\{\exp(h\times T_0)\}=\mathrm{E}\{|\hat{R}_n|^{-hn}\} =\left\{\frac{\Gamma(n/2)}{\Gamma(n/2-hn)}\right\}^{p} \times \frac{\Gamma_{p}(n/2 -hn)}{\Gamma_{p} (n/2)},\notag
\end{eqnarray*} 
where $\Gamma(z)$ denotes the Gamma function, and $\Gamma_p(z)$ denotes the multivariate Gamma function satisfying $	\Gamma_p(z)=\pi^{p(p-1)/4}\prod_{j=1}^p \Gamma\{z-(j-1)/2\}.$ 




\paragraph{Part (i) The chi-square approximation.}
When $p$ is fixed compared to  $N$, by applying Stirling's approximation to the Gamma function, it can be shown that as $N\to \infty$, for any $h\in (-\delta_0, \delta_0)$, $\mathrm{E}\{\exp(h\times T_0)\}$ converges to $(1-2h)^{-f_0/2}$, which is the moment generating function of $\chi_{f_0}^2$; see, e.g., \citet{bartlett1950tests} and  Section 5.1.2 of \citet{muirhead2009aspects}. It follows that $T_0 \xrightarrow{D} \chi^2_{f_0}$ by  the continuity theorem.  
When $p\to \infty$, \citet{jiang2013} and \citet{jiang2015likelihood}  derived an approximate expansion of the multivariate Gamma function $\Gamma_p(\cdot)$  when $p$  increases with the sample size $N$, and then showed that for any $h\in (-\delta_0, \delta_0)$, 
\begin{eqnarray}
	\mathrm{E}[\exp\{ h ( T_0 + n\mu_{n,0} )/( n \sigma_{n,0} ) \} ] \to \exp(h^2/2), \label{eq:normallimit1}
\end{eqnarray} 
where $\delta_0$ is a constant that is sufficiently small, $\exp(h^2/2)$ is the moment generating function of the standard normal random variable $\mathcal{N}(0,1)$, and 
\begin{eqnarray*}
\mu_{n,0} =(p-n+1/2)\log \left(1-\frac{p}{n} \right) - \frac{n-1}{n}p, \quad  	\sigma_{n,0}^2 = - 2\left\{\frac{p}{n}+\log \left(1-\frac{p}{n} \right) \right\}. 
\end{eqnarray*}
This suggests $ ( T_0 + n\mu_{n,0})/( n \sigma_{n,0} ) \xrightarrow{D} \mathcal{N}(0,1)$ by  the continuity theorem.  
Note that $\chi_{f_0}^2$ can be viewed as a summation of the squares of $f_0$ independent standard normal random variables, and  $f_0\to \infty$ when $p\to \infty$. By  applying the central limit theorem to $\chi_{f_0}^2$ when $p\to \infty$, we obtain $(\chi_{f_0}^2 - f_0)/\sqrt{2f_0}\xrightarrow{D} \mathcal{N}(0,1)$, giving $\mathrm{E}[\exp\{h( \chi_{f_0}^2 - f_0 )/ \sqrt{2f_0} \}] \to \exp(h^2/2)$.
Therefore, if the chi-square approximation for $T_0$ holds, we know $\mathrm{E}[\exp\{h( T_0 - f_0 )/ \sqrt{2f_0} \}] \to \exp(h^2/2)$ for $h\in (-\delta_0, \delta_0)$, which, given \eqref{eq:normallimit1}, is equivalent to 
\begin{eqnarray}
	\sqrt{2f_0} \times ( n\sigma_{n,0} )^{-1} \to 1, \label{eq:chisqvar1} \\
	( f_0 + n\mu_{n,0} )\times( n \sigma_{n,0} )^{-1} \to 0. \label{eq:chisqmean1}
\end{eqnarray}
We next examine \eqref{eq:chisqvar1} and \eqref{eq:chisqmean1} by discussing two cases $\lim_{n \to \infty}p/n = 0$ and $\lim_{n \to \infty}p/n = C \in (0,1]$, respectively.

\noindent \textit{Case (i.1): $\lim_{n \to \infty}p/n = 0$.} \  
Under this case, we show that \eqref{eq:chisqvar1} holds. By Taylor's expansion, $\log(1-x)= - x - x^2/2 + O(x^3)$ for $x\in (0,1)$, and then 
\begin{eqnarray}
	\sigma_{n,0}^2=-2\left\{ -\frac{p^2}{2n^2} + O\left(\frac{p^3}{n^3} \right) \right\} = \frac{p^2}{n^2}\{1+o(1)\}. \label{eq:sigmanexp}
\end{eqnarray}
Recall that $f_0=p(p-1)/2$, and it  follows that  \eqref{eq:chisqvar1} holds. Next we prove \eqref{eq:chisqmean1} holds if and only if $p/n^{1/2}\to 0.$ Similarly by Taylor's expansion and $p/n=o(1)$, we have
\begin{eqnarray*}
	\mu_{n,0}&=&(-n+p+1/2)\left\{-\sum_{k=1}^3\frac{1}{k}\left(\frac{p}{n} \right)^k + O\left(\frac{p^4}{n^4} \right) \right\} - \frac{n-1}{n}p \notag \\
	&=& p+\frac{p^2}{2n}+\frac{p^3}{3n^2}-\frac{p(p+1/2)}{n}-\frac{p^2(p+1/2)}{2n^2} + O\left(\frac{p^4}{n^3} \right) -p+ \frac{p}{n} \notag \\
	&=& \frac{p}{2n}-\frac{p^2}{2n}-\frac{p^3}{6n^2} + O\left(\frac{p^4}{n^3} \right)+o\left(\frac{p}{n}\right),
\end{eqnarray*} 
and then $f_0+n\mu_{n,0}=-p^3/(6n)+O(p^4/n^2) + o(p)$. 
 Given that \eqref{eq:chisqvar1} holds under this case and $\sqrt{2f_0}/p\to 1$, we obtain  $	( f_0 + n \mu_{n,0} )\times ( n\sigma_{n,0} )^{-1}= - p^2/(6n) + O(p^3/n^2)+o(1)$,
 which converges to 0 if and only if $p^2/n\to 0$ under this case.   

\smallskip

\noindent \textit{Case (i.2): $\lim_{n \to \infty}p/n = C \in (0,1]$.} \  Under this case, we show that \eqref{eq:chisqvar1} does not hold. Note that 
\begin{eqnarray*}
	\frac{2f_0 }{ n^2\sigma_{n,0}^2} \to \frac{C^2}{-2\{C + \log(1-C)\} }.
\end{eqnarray*}
If $C=1$, $2f_0/(n^2\sigma_{n,0}^2)\to 0$, and thus \eqref{eq:chisqvar1} does not hold. We next consider $C\in (0,1).$ 
If \eqref{eq:chisqvar1} holds, we shall  have  $g_1(C)=0$ with $g_1(C)=C^2 + 2\{C+\log(1-C)\}$. By taking derivative of $g_1(C),$ we obtain
\begin{eqnarray*}
	g'_1(C)=2C+2 -\frac{2}{1-C}=-\frac{2C^2}{1-C} < 0
\end{eqnarray*}when $C\in (0,1).$ 
This suggests that $g_1(C)$ is strictly decreasing on $C\in (0,1)$. As $g_1(0)=0$, we know $g_1(C)<0$ for $C\in (0,1)$, and thus  \eqref{eq:chisqvar1} does not hold. 

\smallskip

Finally, we consider a general sequence $p/n \in (0,1]$, and write $p_n=p$ and $f_{n,0}=f_0$ below to emphasize that $p$ and $f_0$ change with $n$. For the bounded sequence $\{p_n/n\}$, by the Bolzano-Weierstrass theorem, we can further take a subsequence $\{p_{n_k}/n_k\}$ such that $p_{n_k}/n_k \to C \in [0,1].$ If $C\in (0,1]$, the analysis in Case (i.2)  applies, and we know 
$\sqrt{2f_{n_k,0}} \times (n_k\sigma_{n_k,0} )^{-1}$
does not converge to 1. 
 Since a sequence converges if and only if every subsequence converges, we know \eqref{eq:chisqvar1} does not converge to 1 under this case. Alternatively, if all the subsequences of $\{p/n\}$ converge to 0, we know $p/n \to 0,$ and the analysis in Case (i.1) applies. In summary, the chi-square approximation holds if and only if $p^2/n\to 0$.




\paragraph{Part (ii) The chi-square approximation with the Bartlett correction.}
Similarly to the proof of Part (i), when $p$ is fixed, it has been shown that  $\mathrm{E}\{\exp(h\times \rho_0\times T_0)\}\to (1-2h)^{-f_0/2}$ for $h\in (-\delta_0,\delta_0)$ and
$\rho_0=1-(2p+5)/(6n)$
\cite[see, e.g.,][]{bartlett1950tests}; when $p\to \infty$, we also have \eqref{eq:normallimit1} holds. If the chi-square approximation with the Bartlett correction holds, $\mathrm{E}[\exp\{h( \rho_0 T_0 - f_0 )/ \sqrt{2f_0} \}] \to \exp(h^2/2)$ for $h\in (-\delta_0, \delta_0)$, which, given \eqref{eq:normallimit1}, is equivalent to 
\begin{eqnarray}
	\sqrt{2f_0} \times (n\rho_0 \times \sigma_{n,0})^{-1} \to 1, \label{eq:chisqvarbar1} \\
	( f_0 + n\rho_0\times \mu_{n,0} )\times(n\rho_0 \times \sigma_{n,0} )^{-1} \to 0. \label{eq:chisqmeanbar1}
\end{eqnarray}


\noindent \textit{Case (ii.1): $\lim_{n \to \infty}p/n = 0$.} \  Under this case, we have \eqref{eq:chisqvarbar1} holds given $\rho_0\to 1$ and \eqref{eq:chisqvar1} proved above. We next prove \eqref{eq:chisqmeanbar1} holds if and only if $p^3/n^2\to 0.$ Similarly to the proof in Case (i.1), by Taylor's expansion and $p/n=o(1)$, we have
\begin{eqnarray*}
	\mu_{n,0}&=&(-n+p+1/2)\left\{-\sum_{k=1}^4\frac{1}{k}\left(\frac{p}{n} \right)^k + O\left(\frac{p^5}{n^5} \right) \right\} - \frac{n-1}{n}p \notag \\
	&=& p+\frac{p^2}{2n}+\frac{p^3}{3n^2} + \frac{p^4}{4n^3}
	-\frac{p(p+1/2)}{n}-\frac{p^2(p+1/2)}{2n^2} \notag \\
	&&-\frac{p^3(p+1/2)}{3n^3} + O\left(\frac{p^5}{n^4} \right) -p+ \frac{p}{n} \notag \\
	&=& \frac{p}{2n}-\frac{p^2}{2n}-\frac{p^3}{6n^2}-\frac{p^4}{12n^3} + O\left(\frac{p^5}{n^4} \right)+o\left(\frac{p}{n}\right).
\end{eqnarray*} 
By $n\rho_0=n -(2p+5)/6$, we obtain
\begin{align*}
f_0+n\rho_0\times \mu_{n,0}=&~ f_0 + n\times \mu_{n,0} -p\times \mu_0/3 + o(p) \notag \\
=&~ f_0 +  \frac{p-p^2}{2}-\frac{p^3}{6n}-\frac{p^4}{12n^2} + O\left(\frac{p^5}{n^3} \right)+o(p)	+\frac{p^3}{6n} + \frac{p^4}{18n^2}     \notag \\
=&~ -\frac{p^4}{36n^2}+ O\left(\frac{p^5}{n^3} \right)+o(p).
\end{align*}

Given that \eqref{eq:chisqvarbar1} holds under this case and $\sqrt{2f_0}/p\to 1$, 
we obtain 
 $(f_0 + n\rho_0\mu_{n,0})\times(n\rho_0\sigma_{n,0} )^{-1}= - p^3/(36n^2) + O(p^4/n^3)+o(1)$,
 which converges to 0 if and only if $p^3/n^2\to 0$ under this case. 


\smallskip

\noindent \textit{Case (ii.2): $\lim_{n \to \infty}p/n = C \in (0,1]$.} \ 
Under this case, we show that \eqref{eq:chisqvarbar1} does not hold. Note that 
\begin{eqnarray}
	\frac{2f_0 }{ n^2\rho_0^2 \sigma_{n,0}^2} \to \frac{C^2}{-2(1-C/3)^2\{C + \log(1-C)\} }.\label{eq:climit1}
\end{eqnarray}
If $C=1,$ $2f_0/(n^2\rho_0^2\sigma_{n,0}^2)\to 0$ and thus \eqref{eq:chisqvarbar1} does not hold.  
We next consider $C\in (0,1).$
If \eqref{eq:chisqvarbar1} holds, we shall   have  $g_2(C)= 0$ with $g_2(C)=C^2 + 2(1-C/3)^2\{C+\log(1-C)\}$. By taking derivative of $g_2(C),$ we obtain $g'_2(0)=0,$ $g''_2(0)=0$, and
\begin{eqnarray*}
g_2'''(C)&=&	 - \frac{4C(3C^2-8C+9)}{9(1-C)^3} < 0
\end{eqnarray*} when $C\in (0,1)$. Similarly to the analysis in Case (i.1), we obtain that $g_2'(C)<0$ for $C\in (0,1)$. It follows that $g_2(C)$ is strictly  decreasing on $C\in (0,1)$ with $g_2(0)=0$. Therefore $g_2(C)<0$ on  $C\in (0,1)$,  which suggests that   \eqref{eq:chisqvarbar1} does not hold. 


Finally, for a general sequence $p/n \in (0,1]$, following the analysis of taking subsequences in Part (i),  we know that the chi-square approximation with the Bartlett correction holds if and only if $p^3/n^2 \to 0$. Recall that $N=n+1$. Thus, the same conclusions hold asymptotically by replacing $n$ with $N$, that is, the chi-square approximations without and with the Bartlett correction hold if and only if $p^2/N\to 0$ and $p^3/N^2$, respectively.


\subsection{Proof of Theorem \ref{thm2}}\label{sec:pfthm2}

Similarly to the proof of Theorem \ref{thm1}, we next examine the limiting distribution of $T'$ based on its moment generating function.  
In Theorem \ref{thm2}, testing $H_{0,k}': \Sigma = \Lambda_k \Lambda_k^{\top}+ \Psi_k$ when $\Lambda_k$ and $ \Psi_k$ are given is equivalent to testing the null hypothesis $H_0: \Sigma = \mathrm{I}_p$ by applying the data transformation $\Sigma_k^{-1/2} X_i $ with $\Sigma_k= \Lambda_k \Lambda_k^{\top}+ \Psi_k.$ 
Then by Corollary 8.4.8 in \citet{muirhead2009aspects}, under the null hypothesis, we have
\begin{eqnarray}
\mathrm{E}\{\exp(h \times T')\}= 	\left(\frac{2e}{n}\right)^{-pnh}(1-2h)^{-pn(1-2h)/2}\times \frac{\Gamma_p \{n(1-2h)/2\}}{\Gamma_p(n/2)}, \label{eq:momentgencov}
\end{eqnarray} where $n=N-1$.  



\paragraph{Part (i) The chi-square approximation.} When $p$ is fixed compared to the sample size $N$, by applying Stirling's approximation to the Gamma function, it has been shown that as $N\to \infty$, \eqref{eq:momentgencov} converges to $(1-2h)^{-f'/2}$, which is the moment generating function of $\chi^2_{f'}$ \cite[][Section 8.4.4]{muirhead2009aspects}, and therefore  $T' \xrightarrow{D} \chi^2_{f'}$. 
When $p\to \infty$, by the proof of Lemma \ref{lm:1} in Section \ref{sec:lemma}, we have $ \mathrm{E}[\exp\{ h ( T' + n\mu_{n} )/( n \sigma_{n} ) \} ] \to \exp(h^2/2)$, where 
\begin{align}
	\mu_n =  -p + (p-n+1/2) \log \left(1-\frac{p}{n}\right), \quad  
	\sigma_n^2 =-2\left\{\frac{p}{n}+\log \left(1-\frac{p}{n}\right)\right\}. \label{eq:lmmusigma}
\end{align} 
Similarly to the proof of Theorem \ref{thm1}, we know that the chi-square approximation for $T'$ holds if and only if
\begin{eqnarray}
	\sqrt{2f'} \times ( n\sigma_{n} )^{-1} \to 1, \label{eq:chisqvar2} \\
	( f' + n\mu_{n} )\times( n \sigma_{n} )^{-1} \to 0. \label{eq:chisqmean2}
\end{eqnarray}


\noindent \textit{Case (i.1): $\lim_{n \to \infty}p/n = 0$.} \  Under this case, similar to  \eqref{eq:sigmanexp}, by Taylor's expansion, $\sigma_n^2=p^2n^{-2}\{1+o(1)\}$.   
As $\sqrt{2f'}/p\to 1$, we have \eqref{eq:chisqvar2} holds. We next show that \eqref{eq:chisqmean2} holds if and only if $p^2/n\to 0.$ Particularly, by Taylor's expansion and $p/n\to 0$,
\begin{eqnarray}
	\mu_n &=&  -p + (-n+p+1/2) \left\{ -\frac{p}{n}-\frac{p^2}{2n^2}-\frac{p^3}{3n^3} + O\left(\frac{p^4}{n^4} \right)\right\} \label{eq:munexpan2} \\
	&=&-p+p+\frac{p^2}{2n}+\frac{p^3}{3n^2}-\frac{p^2}{n}-\frac{p^3}{2n^2} - \frac{p}{2n}  + O\left(\frac{p^4}{n^3} \right) + o \left( \frac{p}{n}\right). \notag
\end{eqnarray}
It follows that $f'+n\mu_n=-p^3/(6n^2)+O(p^4/n^{3})+o(p/n)$. Given \eqref{eq:chisqvar2} and $\sqrt{2f'}/p\to 1$,  \eqref{eq:chisqmean2} holds if and only if $p^2/n \to 0.$


\noindent \textit{Case (i.2): $\lim_{n \to \infty}p/n = C \in (0,1]$.} \ Under this case,  $2f'/(n^2\sigma_n^2) \to -C^2/[2\{C+\log(1-C)\}]$.  We then know \eqref{eq:chisqvar2} does not hold following the proof of Theorem \ref{thm1}, and therefore the chi-square approximation fails.

Finally, for a general sequence $p/n \in (0,1]$, following the same analysis of taking subsequences as in the proof of Theorem \ref{thm1},  we know that the chi-square approximation holds if and only if $p^2/n \to 0$. 

\paragraph{Part (ii) The chi-square approximation with the Bartlett correction.}
Similarly to the proof of Theorem \ref{thm1} and the analysis above, we know that the chi-square approximation with the Bartlett correction holds if and only if
\begin{eqnarray}
	\sqrt{2f'} \times (n\rho' \times \sigma_{n})^{-1} \to 1, \label{eq:chisqvarbar2} \\
	( f' + n\rho' \times \mu_{n} )\times(n\rho' \times \sigma_{n} )^{-1} \to 0. \label{eq:chisqmeanbar2}
\end{eqnarray}

\noindent \textit{Case (ii.1): $\lim_{n \to \infty}p/n = 0$.} \  As $\rho'\to 1$ under this case, we know \eqref{eq:chisqvarbar2} holds given \eqref{eq:chisqvar2} proved in Part (i). We next prove \eqref{eq:chisqmeanbar2} holds if and only if $p^3/n^2 \to 0.$ Similarly to \eqref{eq:munexpan2}, by Taylor's expansion and $p/n\to 0$,
\begin{eqnarray*}
	\mu_n&=&  -p + (-n+p+1/2) \left\{ -\sum_{j=1}^4 \frac{p^j}{j\times n^j} + O\left(\frac{p^5}{n^5} \right)\right\} \notag \\
	&=&-\frac{p(p+1)}{2n}-\frac{p^3}{6n^2}-\frac{p^4}{12n^3} + O\left( \frac{p^5}{n^4}\right) +o\left(\frac{p}{n} \right). 
\end{eqnarray*}
By $n\rho'=n-p/3+O(1)$ and $p/n\to 0,$  
\begin{align*}
n\rho'\mu_n =&~ \{n-p/3+O(1)\}\left\{-\frac{p(p+1)}{2n}-\frac{p^3}{6n^2}-\frac{p^4}{12n^3} + O\left( \frac{p^5}{n^4}\right) +o\left(\frac{p}{n} \right) \right\} + o(p)
\notag \\
=&~ -\frac{p(p+1)}{2}-\frac{p^3}{6n}-\frac{p^4}{12n^2} +\frac{p^3}{6n}+\frac{p^4}{18n^2} +O\left( \frac{p^5}{n^3}\right) + o(p).
\end{align*}
It follows that $f'+n\rho'\mu_n = - p^4/(36n^2)+O(p^4/n^3)+o(p)$. Given \eqref{eq:chisqmeanbar1} and $\sqrt{2f'}/p\to 1$, \eqref{eq:chisqmeanbar2} holds if and only if $p^3/n^2 \to 0.$ 

\noindent \textit{Case (ii.2): $\lim_{n \to \infty}p/n = C \in (0,1]$.} \ Under this case, $\rho'\to 1-C/3$ and $2f'/(n\rho'\sigma_n)^2$ converges to the limit same as the right hand side of \eqref{eq:climit1}. Thus the same analysis applies and we know that the chi-square approximation with the Bartlett correction fails. 

Finally, for a general sequence $p/n \in (0,1]$, following the same analysis of taking subsequences as in the proof of Theorem \ref{thm1},  we know that the chi-square approximation with the Bartlett correction holds if and only if $p^3/n^2 \to 0$. 
Recall that $N=n+1$. Thus, the same conclusions hold asymptotically by replacing $n$ with $N$, that is, the chi-square approximations without and with the Bartlett correction hold if and only if $p^2/N\to 0$ and $p^3/N^2$, respectively. 
 


\subsection{Lemma} \label{sec:lemma}
\begin{lemma}\label{lm:1}
Under the conditions of Theorem \ref{thm2}, when $p\to \infty$ as $n = N-1\to \infty$, we have 
$(T' + n\mu_n)/(n\sigma_n) \xrightarrow{D} \mathcal{N}(0,1)$ with $\mu_n$ and $\sigma_n^2$ in  \eqref{eq:lmmusigma}.
\end{lemma}
\begin{proof}
It suffices to show that there exists a constant $\delta'>0$ such that $ \mathrm{E}[\exp\{ h ( T' + n\mu_{n} )/( n \sigma_{n} ) \} ] \to \exp(h^2/2)$ for all $|h|<\delta'$. Particularly, we let $s=h/(n\sigma_n)$, and prove $\log [\mathrm{E}\{\exp(sT')\}]\to h^2/2 - h \mu_n /\sigma_n$. By the moment generating function of $T'$ in \eqref{eq:momentgencov}, we have 
\begin{eqnarray}
&& \log \left[ \mathrm{E}\{\exp(s \times T')\} \right]\label{eq:momentexp2limit} \\
&=& -pns\log(2e/n)-\frac{pn}{2}(1-2s)\log(1-2s)+\log\left\{ \frac{\Gamma_p (n/2 - ns)}{\Gamma_p(n/2)} \right\}. \notag
\end{eqnarray}
We next derive the approximate expansion of \eqref{eq:momentexp2limit} by discussing two cases. 



\noindent \textit{Case 1: $\lim p/n \to C \in (0,1]$.}\quad   Under this case, we utilize the approximate expansion of multivariate gamma function in Lemma 5.4 of \citet{jiang2013}. 
To apply the result, 
we first show that the conditions are satisfied. Specifically, define $r_n^2=- \log (1-p/n) $, and we have 
\begin{eqnarray*}
	(-ns)^2\times r_n^2= - \frac{h^2}{\sigma_n^2} \log (1-p/n) \to  
\begin{cases}
  \displaystyle \frac{h^2}{2} \times \frac{\log(1-C)}{C+\log(1-C)}, & \text{if } C \in (0,1);  \\[10pt]
\displaystyle      \frac{h^2}{2}, & \text{if } C = 0.
\end{cases}   
\end{eqnarray*}Therefore, $-ns=O(1/r_n)$, and then Lemma 5.4 in \citet{jiang2013} can be applied to expand \eqref{eq:momentexp2limit}.
It follows that 
\begin{eqnarray*}
	\eqref{eq:momentexp2limit} &=&  -pns\log(2e/n)-\frac{pn}{2}(1-2s)\log(1-2s) \notag \\
	&& -pns\log\{n/(2e)\}+r_n^2\left\{ (-ns)^2 -(p-n+1/2)(-ns) \right\} + o(1). 
\end{eqnarray*}
By Taylor's expansion $(1-2s)\log(1-2s)=-2s+2s^2+O(s^3)$ for $s\in (0,1)$, we obtain
\begin{eqnarray*}
\eqref{eq:momentexp2limit} &=& -\frac{pn}{2}	\left\{ -2s+2s^2 + O(s^3)\right\} \notag \\
&& -\log\left(1-\frac{p}{n}\right)\left\{ n^2s^2+(p-n+1/2) ns \right\} + o(1) \notag \\
&=& s^2 \left\{ -pn-n^2\log\left(1-\frac{p}{n}\right) \right\} + s\left\{ pn-(p-n+1/2)\log\left( 1-\frac{p}{n}\right) \right\} + o(1).
\end{eqnarray*}
With $s=h/(n\sigma_n)$, we have  $\log(\mathrm{E}[\exp\{ hT'/(n\sigma_n)\}])= h^2/2 - h \mu_n /\sigma_n  + o(1)$. 

\medskip

\noindent \textit{Case 2: $\lim p/n = 0$.}\quad Under this case, we utilize the approximate expansion of multivariate gamma function in Proposition A.1 of \citet{jiang2015likelihood}. 
To apply the result, 
we first show that the conditions are satisfied. Particularly, as $\sigma_n^2=p^2n^{-2}\{1+o(1)\}$, we have $-ns \times p / n = -ph(n\sigma_n)^{-1}= h \{ 1+ o(1)\}$. Therefore, $-ns=O(n/p)$, and we can apply  Proposition A.1 in \citet{jiang2015likelihood} to expand \eqref{eq:momentexp2limit}. It follows that 
\begin{eqnarray*}
	\log\left\{ \frac{\Gamma_p (n/2 - ns)}{\Gamma_p(n/2)} \right\}=\gamma_{n,1}(-ns)+\gamma_{n,2}(-ns)^2+ \gamma_{n,3}+o(1),
\end{eqnarray*} where 
\begin{align*}
	\gamma_{n,1}=&~-\left\{2p+(n-p-1/2)\log\left(1-{p}/{n} \right) \right\}, \notag \\
	\gamma_{n,2}=&~ -\left\{{p}/{n}+\log\left(1-{p}/{n} \right) \right\},\notag \\
	\gamma_{n,3}=&~ p\left\{\left({n}/{2}-ns\right)\log\left({n}/{2}-ns\right)-(n/2)\log\left(n/2\right) \right\}. \notag 
\end{align*}
Note that $	\gamma_{n,3}=({pn}/{2})(1-2s)\log(1-2s)-pns\log(n/2).$ Then we have
\begin{align*}
 \eqref{eq:momentexp2limit}
=&~-pns\log\left(\frac{2e}{n}\right)-\frac{pn}{2}(1-2s)\log(1-2s) -\gamma_{n,1}ns+\gamma_{n,2}n^2s^2+\gamma_{n,3}+o(1) \notag \\
=&~ -(p+\gamma_{n,1})ns+\gamma_{n,2}n^2s^2+o(1),  \notag 
\end{align*}
which gives $\log(\mathrm{E}[\exp\{ hT'/(n\sigma_n)\}])=h^2/2+\mu_nh/\sigma_n+o(1)$ by $s=h/(n\sigma_n)$.


\smallskip

Finally, for a general sequence $\{p/n\}$, to prove that $(T'+n\mu_n)/(n\sigma_n)$ converges in distribution to $\mathcal{N}(0,1)$, 
it suffices to show that every subsequence has a further subsequence that converges in distribution to $\mathcal{N}(0,1)$. By the boundedness of $p/n$ and the  Bolzano-Weierstrass theorem, we can further take a subsequence such that $p/n$ has a limit and the arguments above can be applied. In summary, Lemma \ref{lm:1} is proved. 
\end{proof}


\section{Supplementary simulation studies}  \label{sec:appendixsimu}


\subsection{Simulations on the Type \RNum{1} Error}\label{sec:addsimu}
In this section, we provide additional simulation studies when the data is not normally  distributed. Particularly, we focus on the likelihood ratio test under the null hypothesis $H_{0,0}$, which detects the existence of any factors or not. 

\paragraph{Simulations with heavy-tailed $\boldsymbol{t}$-distributed data.}

Similarly to previous simulations, 
we consider  $p = \lfloor{N^{\varepsilon}}\rfloor$, where $N \in \{100, 500, 1000, 2000\}$ and $\varepsilon \in \{3/24,4/24, \ldots ,23/24\}$.
Under each combination of $(N,p)$, we generate the entries of data matrix $X_i$ as independent and identical random variables following $t_{d_0}$ distribution, where $d_0$ denotes the degrees of freedom and we take $d_0\in \{5,10\}.$
Then we conduct the likelihood ratio test for $H_{0,0}$ with approximations \eqref{eq1} and \eqref{eq2}. We repeat the procedure 1000 times, and estimate the type \RNum{1} error rates with significance level 0.05. 
We present the results of $t_5$ and $t_{10}$ distributed data in  Figures \ref{fig:typeierrort} and \ref{fig:typeierrortdf10}, respectively. In each figure, we draw the estimated type \RNum{1} error rates versus $\varepsilon$ values for approximations \eqref{eq1} and \eqref{eq2} in the left and right plots, respectively. 
Similarly to Numerical Example \ref{numex:2}, we can see that the chi-square approximation for $T_0$ starts to fail when  $\varepsilon$ approaches $1/2$, and the  chi-square approximation for $\rho_0T_0$ starts to fail when  $\varepsilon$ approaches $2/3$.



\begin{figure}[!ht]
\captionsetup[subfigure]{labelformat=empty}
\centering
\begin{subfigure}[b]{0.4\textwidth}
\centering
\caption{\quad \quad  \small{Approximation \eqref{eq1} for $T_0$}}
\includegraphics[width=\textwidth]{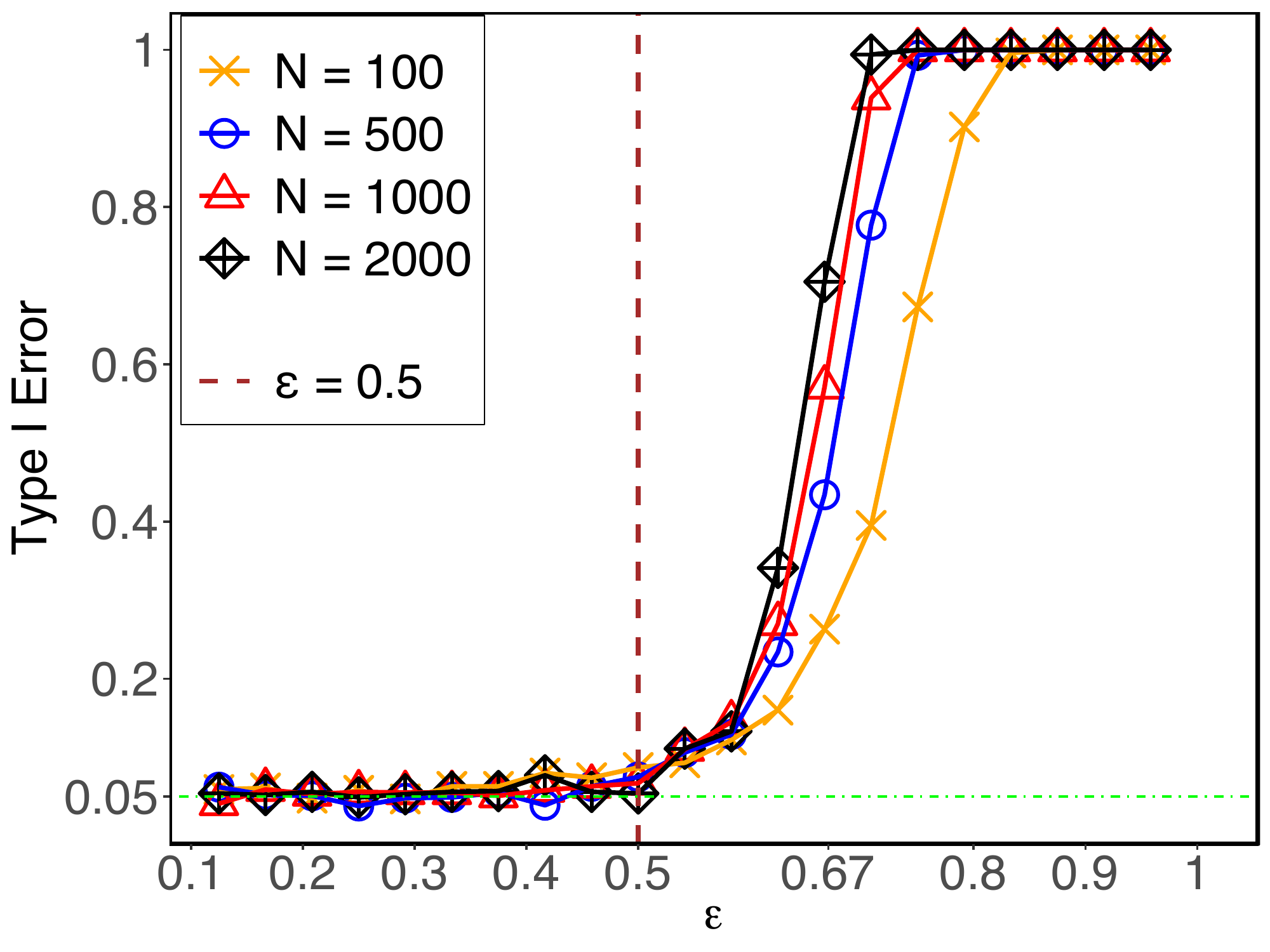}	
\end{subfigure}\quad \quad  
\begin{subfigure}[b]{0.4\textwidth}
\centering
\caption{\quad \quad \small{Approximation \eqref{eq2} for $\rho_0T_0$}}
\includegraphics[width=\textwidth]{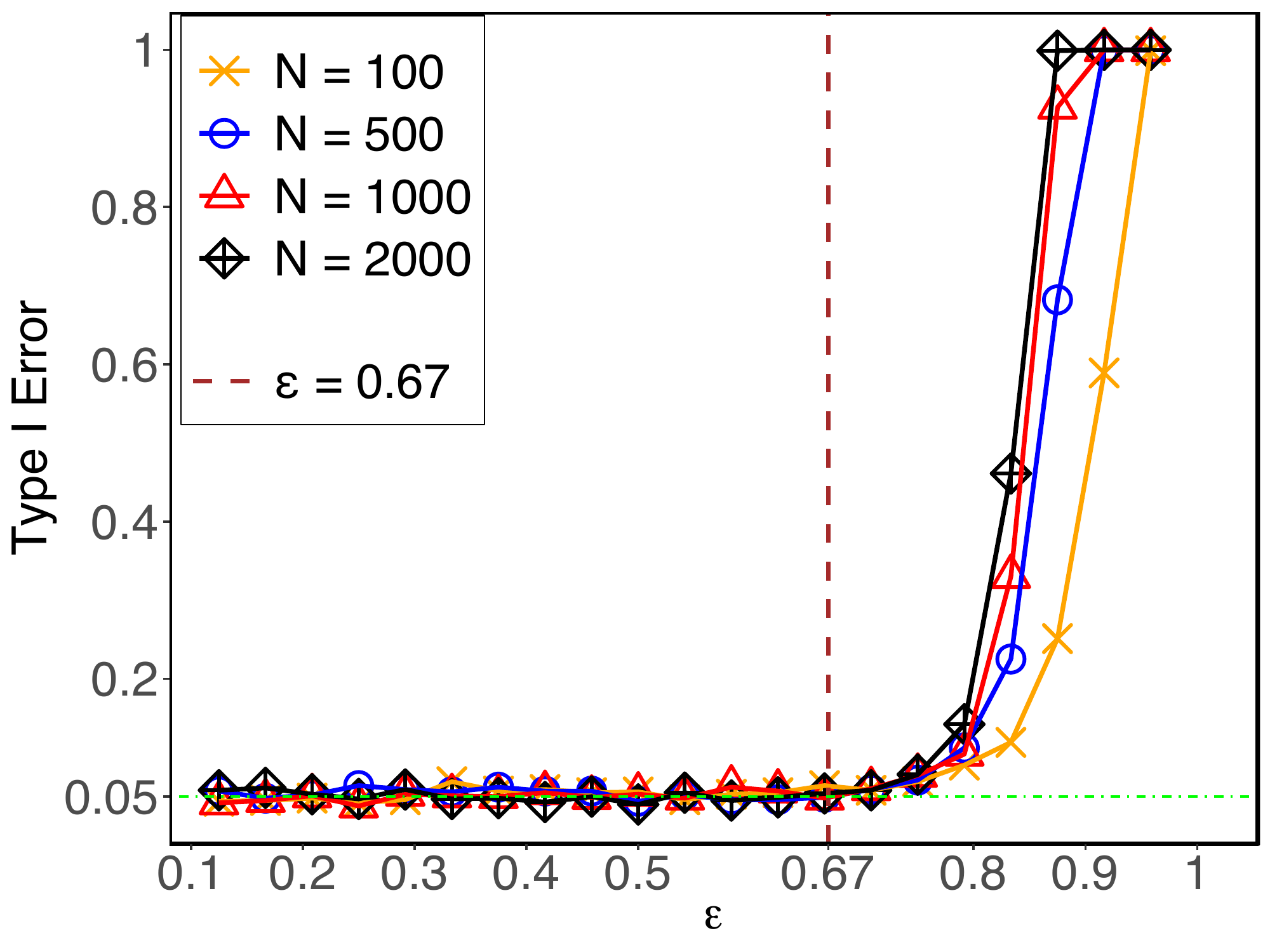}	
\end{subfigure}
\caption{Estimated type \RNum{1} error versus $\varepsilon$ of $t_5$-distributed data}\label{fig:typeierrort}
\end{figure}

\begin{figure}[!ht]
\captionsetup[subfigure]{labelformat=empty}
\centering
\begin{subfigure}[b]{0.4\textwidth}
\centering
\caption{\quad \quad  \small{Approximation \eqref{eq1} for $T_0$}}
\includegraphics[width=\textwidth]{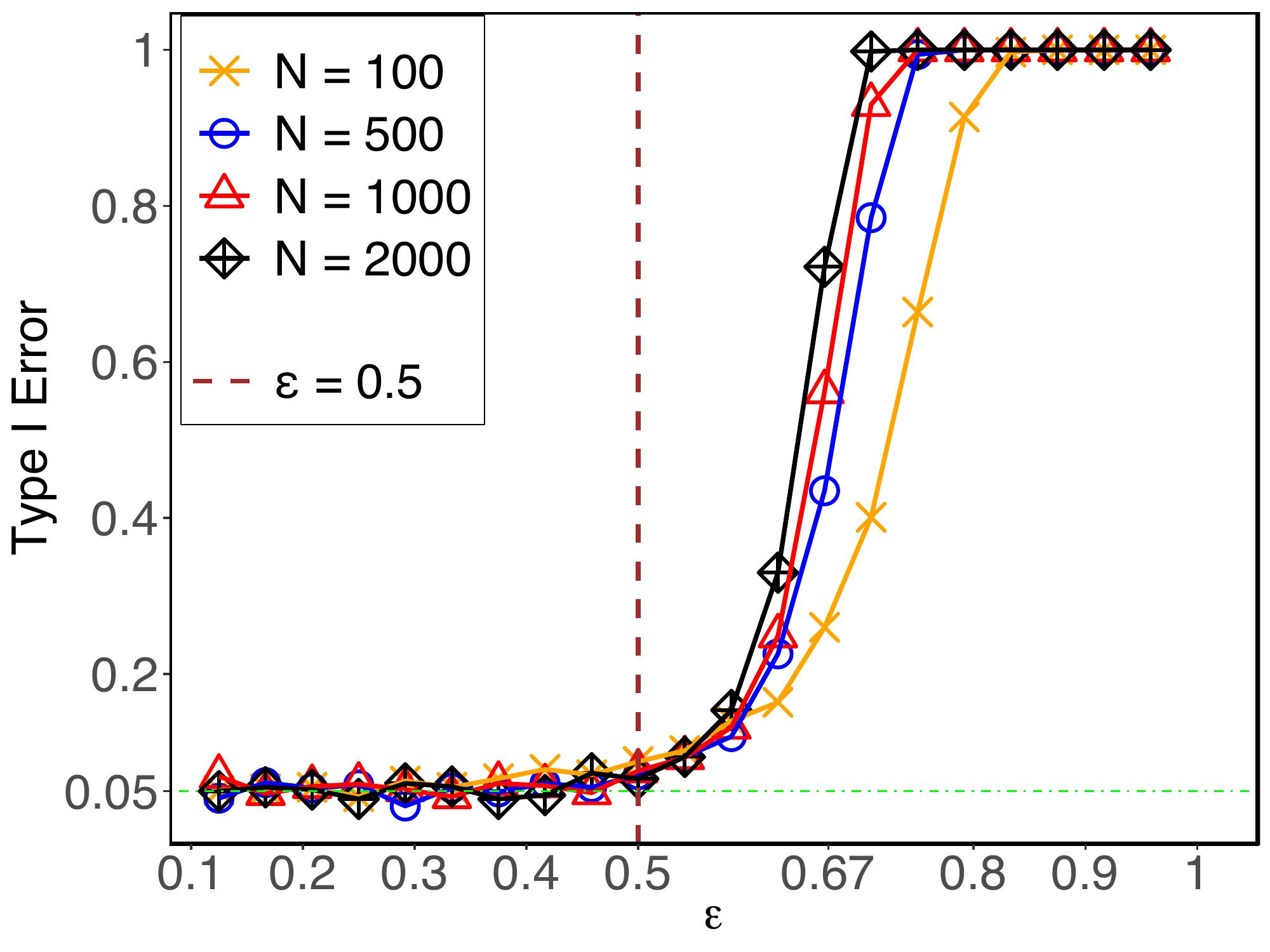}	
\end{subfigure}\quad \quad  
\begin{subfigure}[b]{0.4\textwidth}
\centering
\caption{\quad \quad \small{Approximation \eqref{eq2} for $\rho_0T_0$}}
\includegraphics[width=\textwidth]{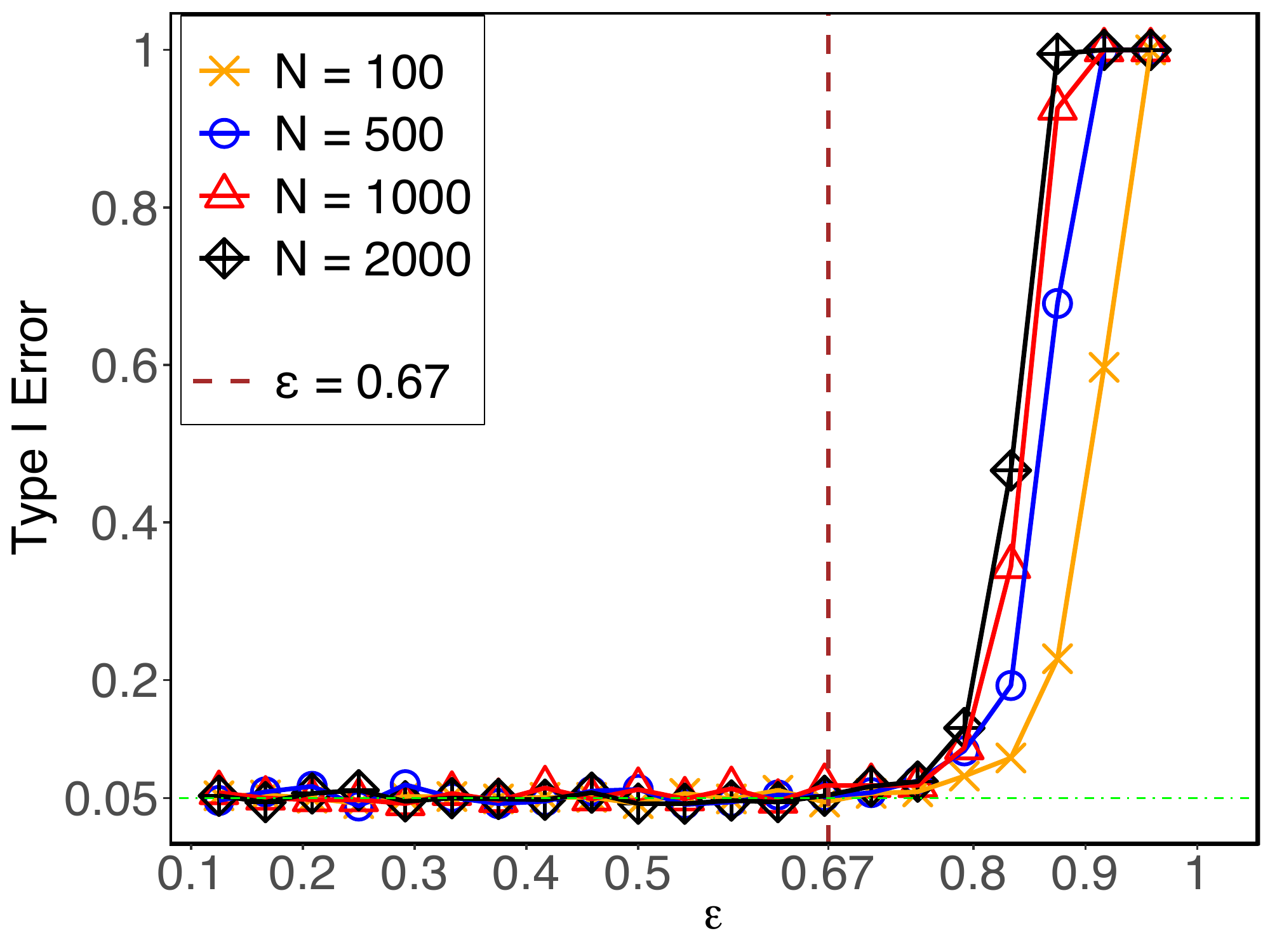}	
\end{subfigure}
\caption{Estimated type \RNum{1} error versus $\varepsilon$ of $t_{10}$-distributed data}\label{fig:typeierrortdf10}
\end{figure}

%



\paragraph{Simulations with discrete multinomial data.}
The simulations are conducted same as above, except that we generate the entries in the data matrix $X_i$ from a discrete multinomial distribution. 
Specifically, for each entry $x_{i,j}$ within the matrix $X_i$, where $i=1,\ldots, N$ and $j=1,\ldots, p$, we first sample $z_{i,j}\sim \mathcal{N}(0,1)$, and then set discrete value of $x_{i,j}$ according to the range of $z_{i,j}$ considering  three settings (I)--(III) in Table \ref{table:discretsetting}. The results under settings (I)--(III) are given in Figures \ref{fig:typeierrormulti}--\ref{fig:typeierrormulti_case3}, respectively. 	Similarly to Numerical Example \ref{numex:2}, under each setting, we observe that the chi-square approximation \eqref{eq1} for $T_0$ starts to fail when  $\varepsilon$ approaches $1/2$, and the  chi-square approximation \eqref{eq2} for $\rho_0T_0$ starts to fail when  $\varepsilon$ approaches $2/3$. 



\begin{table}
\centering
\renewcommand{\arraystretch}{1}
\setlength{\tabcolsep}{11pt}
 \begin{tabular}{l|cccccc}
\multicolumn{7}{c}{Setting (I)} \\ \midrule
$z_{i,j}$ & $(-\infty, 0)$ & $[0,\infty )$ &  &  & &  \\     
$x_{i,j}$ &  -1 &  1 &  &  &  &  \\ 
\midrule
\multicolumn{7}{c}{Setting (II)} \\ \midrule
 $z_{i,j}$  & $(-\infty, -1)$ & $[-1,0)$ &  $ [0,1)$ & $[1,\infty)$ & & \\ 
$x_{i,j}$ &  -2 &  -1 & 1 & 2 & &  \\  
\midrule

\multicolumn{7}{c}{Setting (III)} \\ \midrule
$z_{i,j}$ & $(-\infty, -1)$ & $[-1,-0.4)$ & $[-0.4,0)$ & $ [0,0.4)$ & $[0.4,1)$ & $[1,\infty)$ \\     
$x_{i,j}$ &  -3 &  -2  & -1 & 1 & 2 & 3 \\ 
\midrule
\end{tabular}
\caption{Three settings of correspondence between $x_{i,j}$ and $z_{i,j}$}\label{table:discretsetting}
\end{table}

\begin{figure}[!ht]
\captionsetup[subfigure]{labelformat=empty}
\centering
\begin{subfigure}[b]{0.4\textwidth}
\centering
\caption{\quad \quad  \small{Approximation \eqref{eq1} for $T_0$}}
\includegraphics[width=\textwidth]{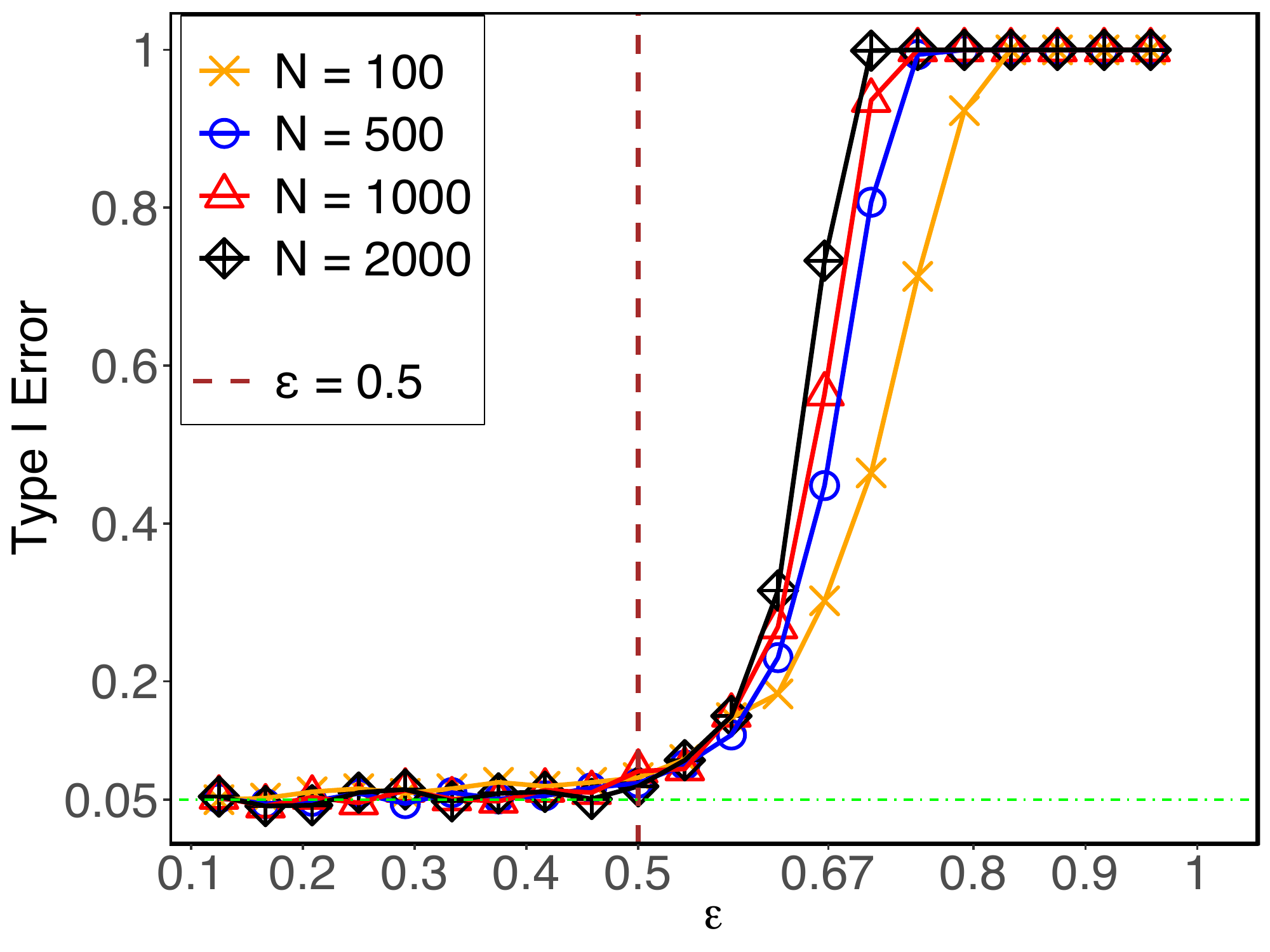}	
\end{subfigure}\quad \quad 
\begin{subfigure}[b]{0.4\textwidth}
\centering
\caption{\quad \quad \small{Approximation \eqref{eq2} for $\rho_0T_0$}}
\includegraphics[width=\textwidth]{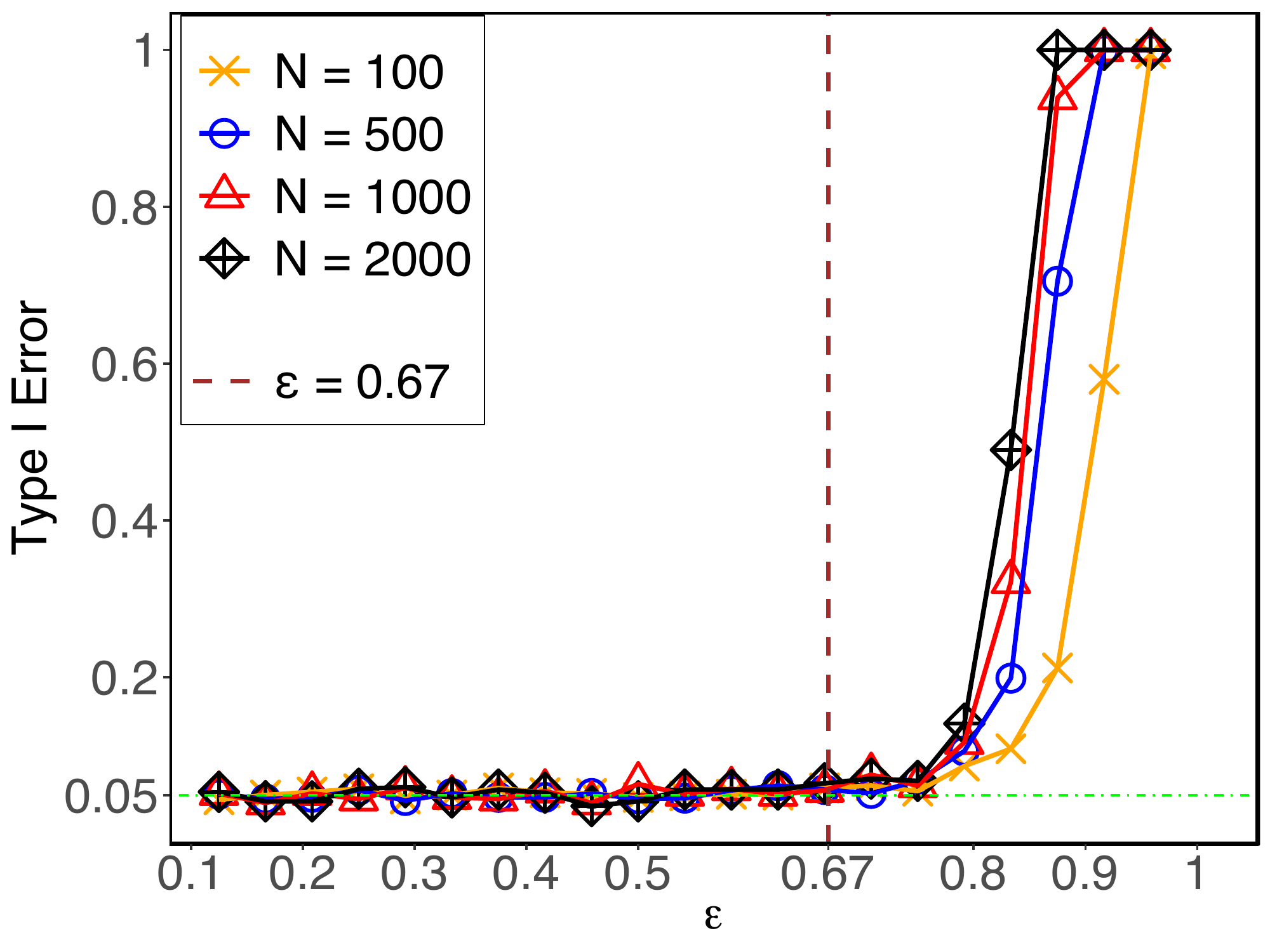}	
\end{subfigure}
\caption{Discrete data (I): Estimated type \RNum{1} error versus $\varepsilon$}\label{fig:typeierrormulti}	
\end{figure}

\begin{figure}[!ht]
\captionsetup[subfigure]{labelformat=empty}
\centering
\begin{subfigure}[b]{0.4\textwidth}
\centering
\caption{\quad \quad  \small{Approximation \eqref{eq1} for $T_0$}}
\includegraphics[width=\textwidth]{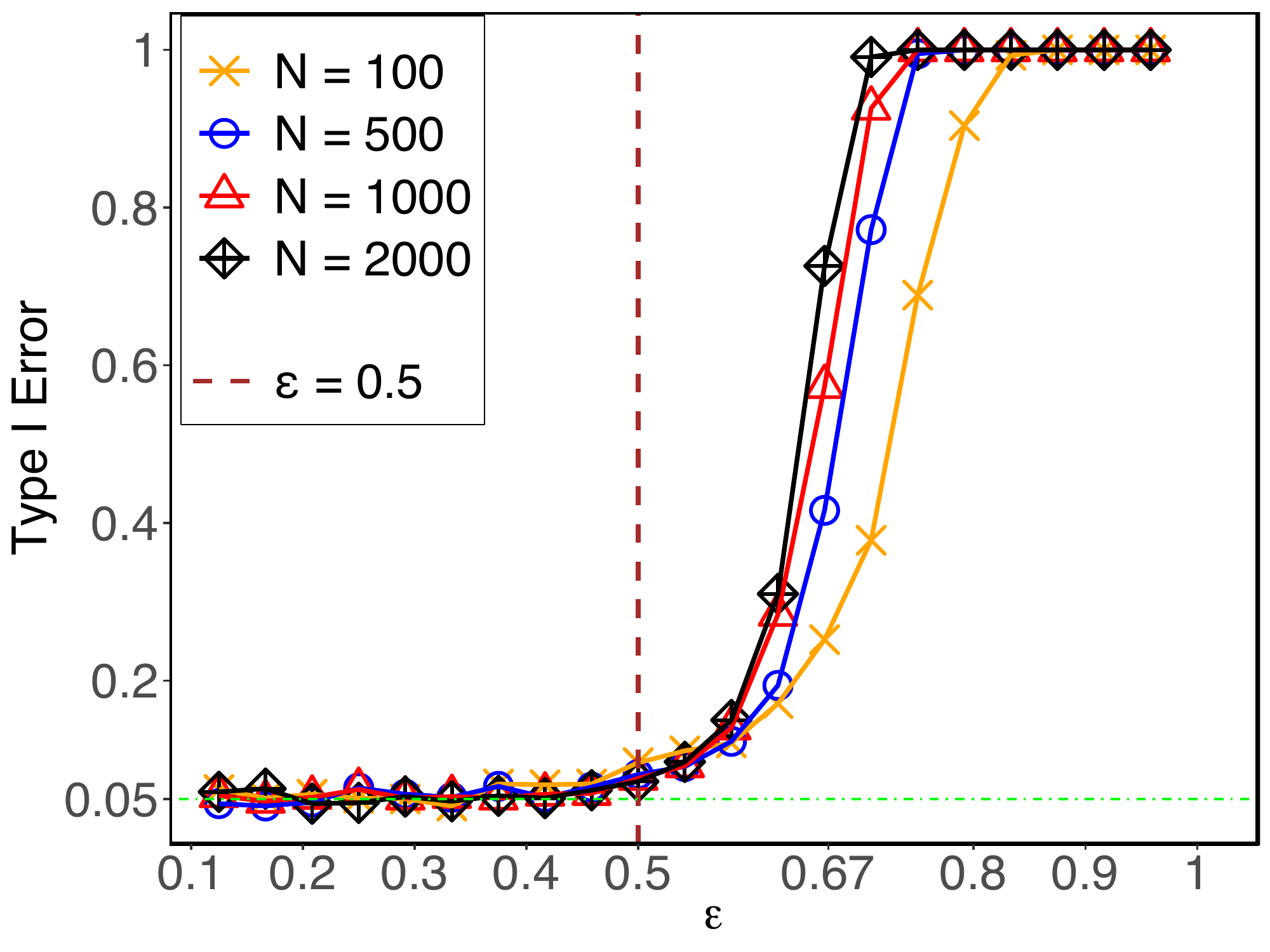}	
\end{subfigure}\quad \quad 
\begin{subfigure}[b]{0.4\textwidth}
\centering
\caption{\quad \quad \small{Approximation \eqref{eq2} for $\rho_0T_0$}}
\includegraphics[width=\textwidth]{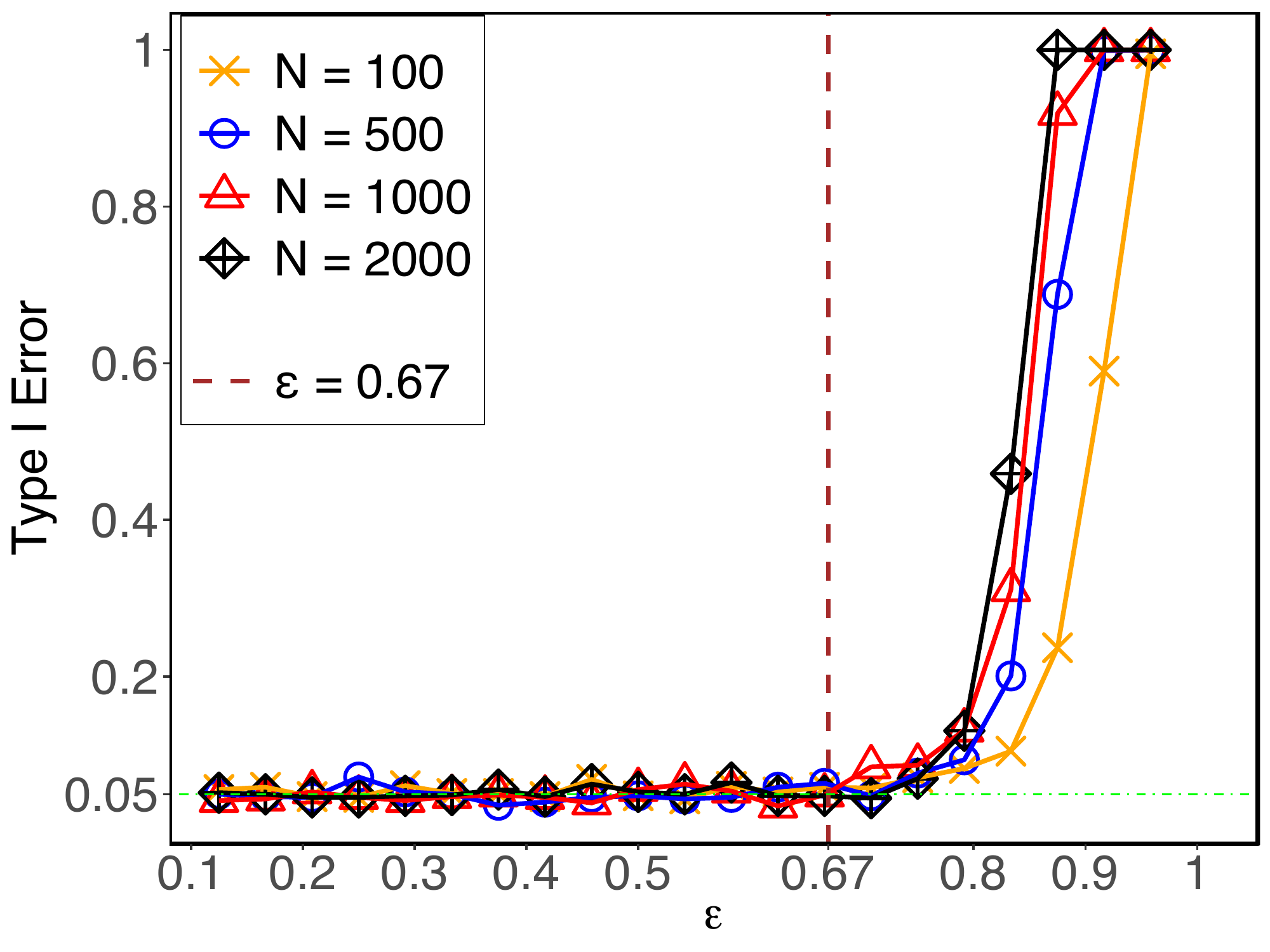}	
\end{subfigure}
\caption{Discrete data (II): Estimated type \RNum{1} error versus $\varepsilon$}\label{fig:typeierrormulti_case2}	
\end{figure}

\begin{figure}[!ht]
\captionsetup[subfigure]{labelformat=empty}
\centering
\begin{subfigure}[b]{0.4\textwidth}
\centering
\caption{\quad \quad  \small{Approximation \eqref{eq1} for $T_0$}}
\includegraphics[width=\textwidth]{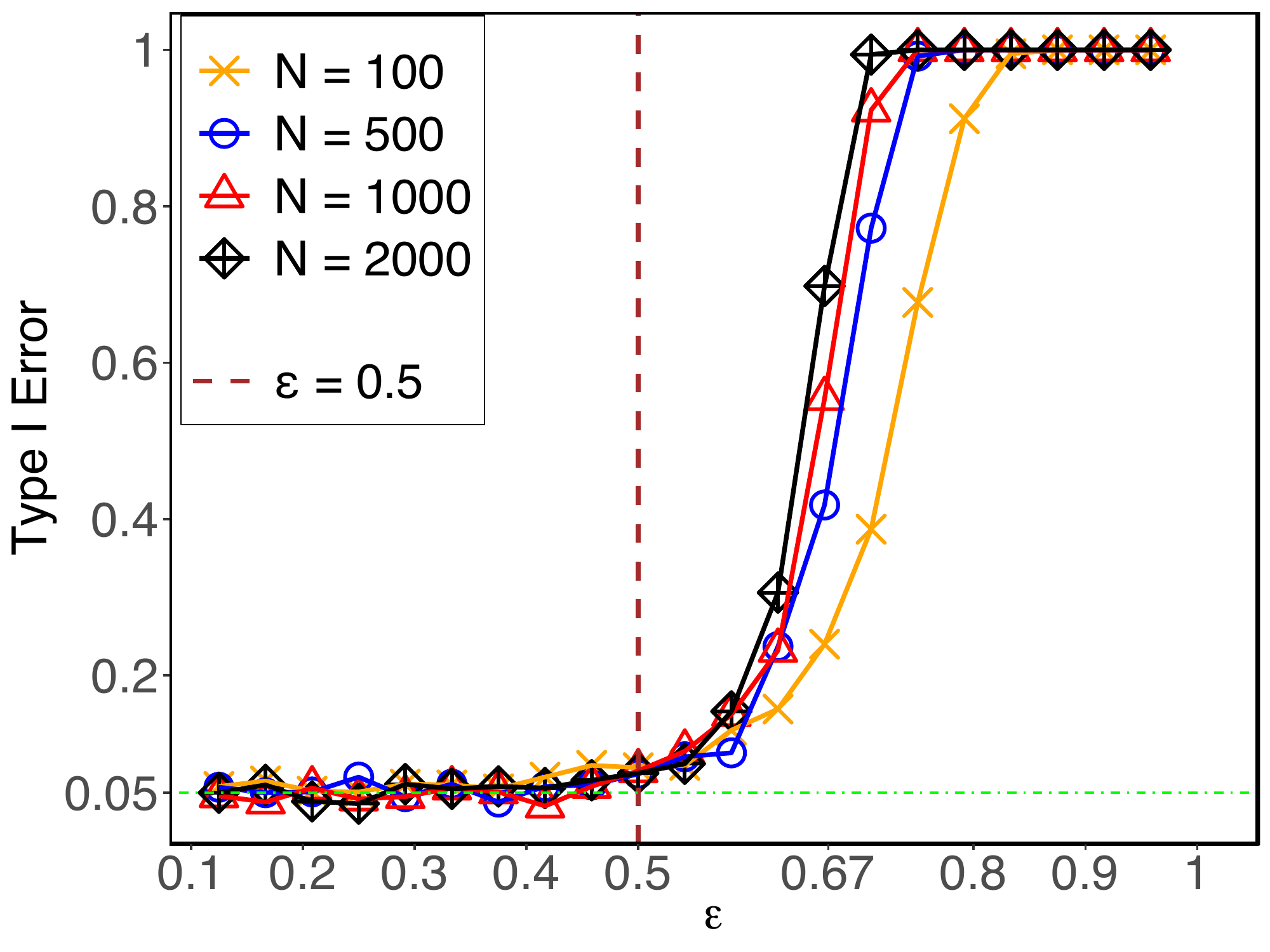}	
\end{subfigure}\quad \quad 
\begin{subfigure}[b]{0.4\textwidth}
\centering
\caption{\quad \quad \small{Approximation \eqref{eq2} for $\rho_0T_0$}}
\includegraphics[width=\textwidth]{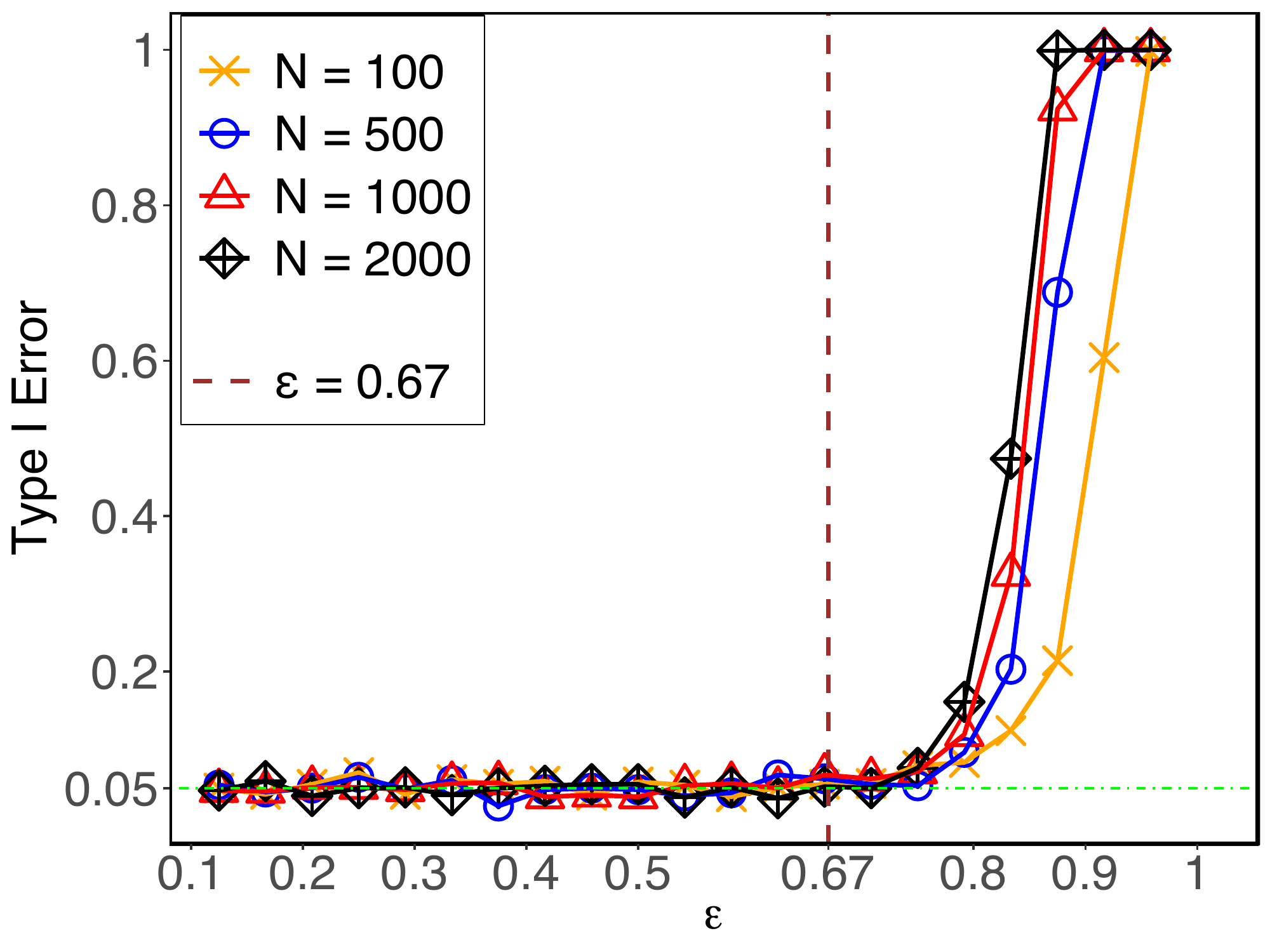}	
\end{subfigure}
\caption{Discrete data (III): Estimated type \RNum{1} error versus $\varepsilon$}\label{fig:typeierrormulti_case3}	
\end{figure}

\newpage

\subsection{Simulations on Estimating the Number of Factors}\label{sec:simfactor}
In this section, we demonstrate the performance of estimating the number of factors using the sequential procedure described in Section \ref{sec2.1}.  
In particular, 
we consider the simulation setting similar to that in Numerical Example \ref{numex:3}, where we take the true number of factors $k_0\in \{1,3\}$, sample size $N\in \{500, 1000\}$ and data dimension $p=\lfloor N^{\epsilon}\rfloor$ for different $\epsilon$ values. 
When conducting the likelihood ratio tests in the sequential procedure, the nominal significance level is set as $\alpha=0.05$. 
For each combination of $(k_0, N)$,
we use the sequential procedure to estimate the number of factors, denoted as $\hat{k}$.  
We repeat the procedure 1000 times and estimate the proportions of correct estimation $(\hat{k}=k_0)$ and overestimation $(\hat{k}>k_0)$, respectively.
We present the results for $k_0=1, 3$ in Figures \ref{fig:factorsel1} and \ref{fig:factorsel2}	, respectively,
where the results based on the likelihood ratio test without and with the Bartlett correction  are given in the left and right columns, respectively. 


The numerical results in Figures \ref{fig:factorsel1} and \ref{fig:factorsel2} show that 
(I) using the likelihood ratio test, 
the procedure begins to overestimate the number of factors when $\epsilon$ approaches $1/2$; 
(II) using the likelihood ratio test with the Bartlett correction, the procedure begins to overestimate the number of factors when $\epsilon$  approaches $2/3$. 
These  observations, compared with Figures \ref{fig:typeierror1}--\ref{fig:typeierror3}, suggest that the sequential procedure begins to overestimate the number of factors when the corresponding type \RNum{1} error begins to inflate, which is consistent with our discussions in Section \ref{sec2.2}. 
Moreover, in Figures \ref{fig:factorsel1} and \ref{fig:factorsel2}, when $\epsilon$ is small and does not  pass the corresponding  phase transition boundary, the proportion of overestimation $(\hat{k}>k_0)$ is around 0.05.  
This is because that rejecting $H_{0,k_0}$ suggests $\hat{k}>k_0$, and 
 the probability of rejecting $H_{0,k_0}$ (type \RNum{1} error of testing $H_{0,k_0}$) can be asymptotically  controlled at the level $\alpha=0.05$ under the asymptotic regimes derived in Theorems \ref{thm1} and \ref{thm2}.      




\begin{figure}[!ht]
\captionsetup[subfigure]{labelformat=empty}
\centering
\begin{subfigure}[b]{0.4\textwidth}
\centering
\caption{\quad  $N=500$; No Correction}
\includegraphics[width=\textwidth]{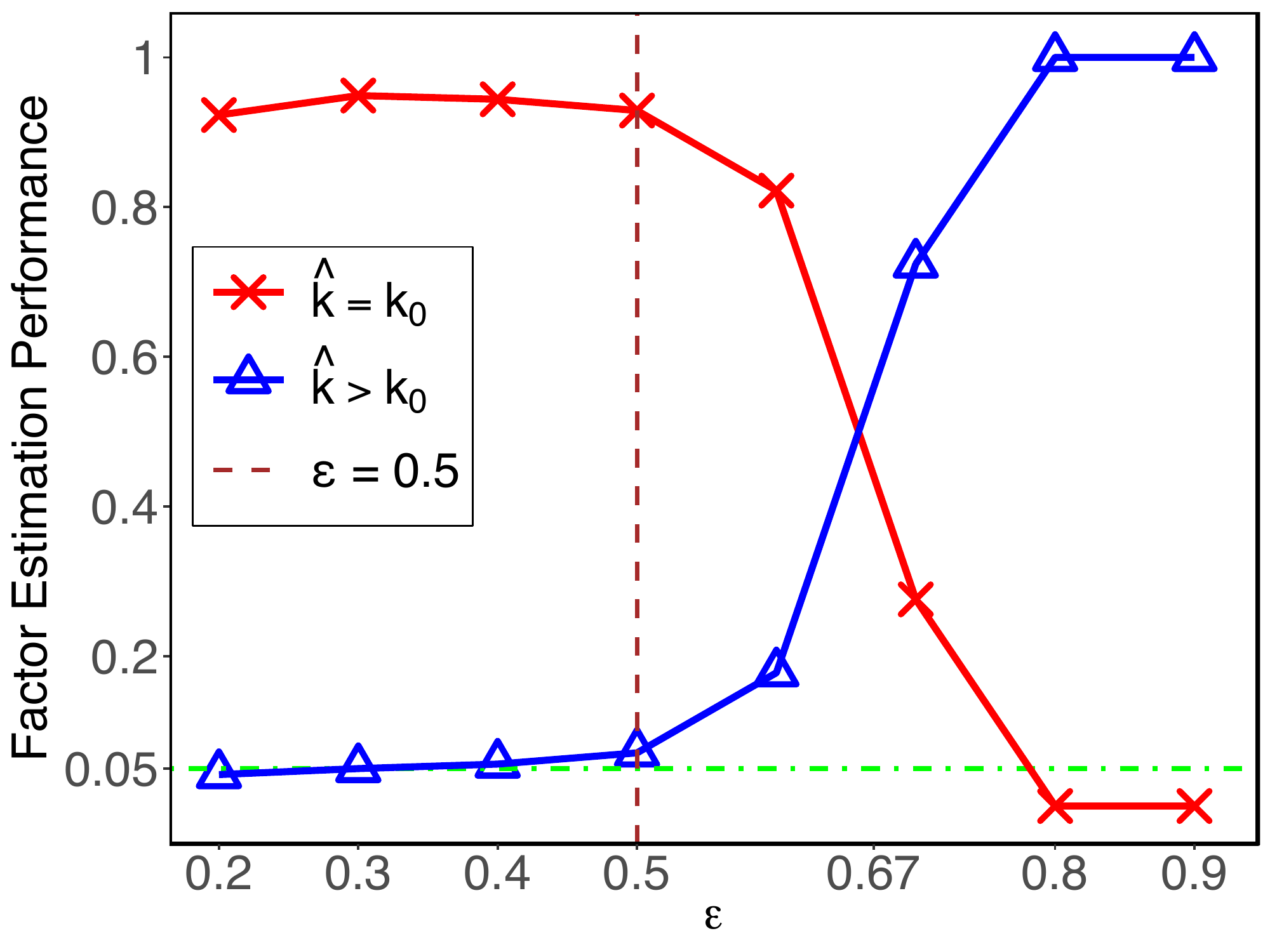}	
\end{subfigure}\quad \quad \
\begin{subfigure}[b]{0.4\textwidth}
\centering
\caption{\quad $N=500$; Bartlett Correction}
\includegraphics[width=\textwidth]{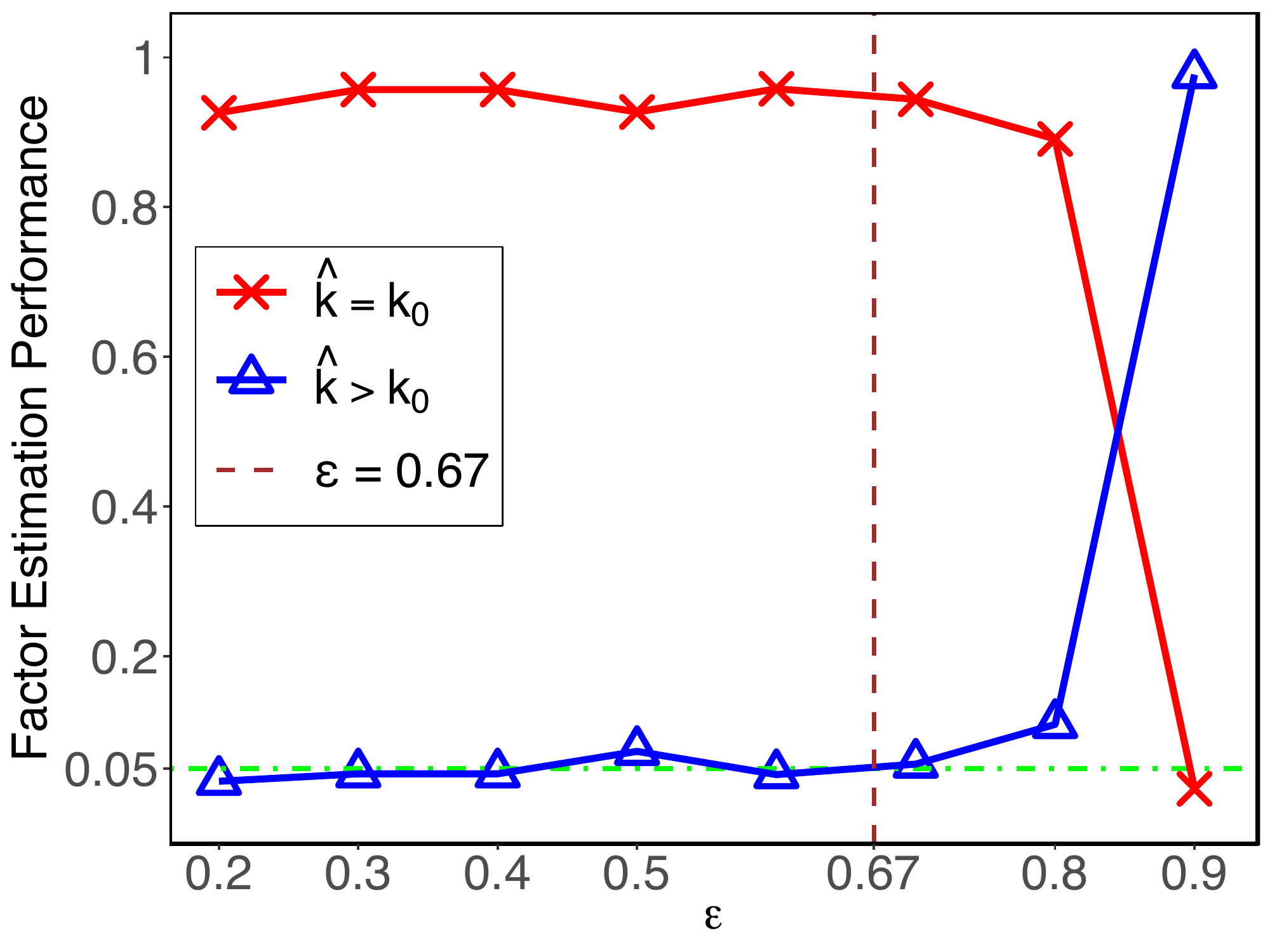}	
\end{subfigure}

\quad \\[-6pt]
-
\begin{subfigure}[b]{0.4\textwidth}
\centering
\caption{\quad $N=1000$; No Correction}
\includegraphics[width=\textwidth]{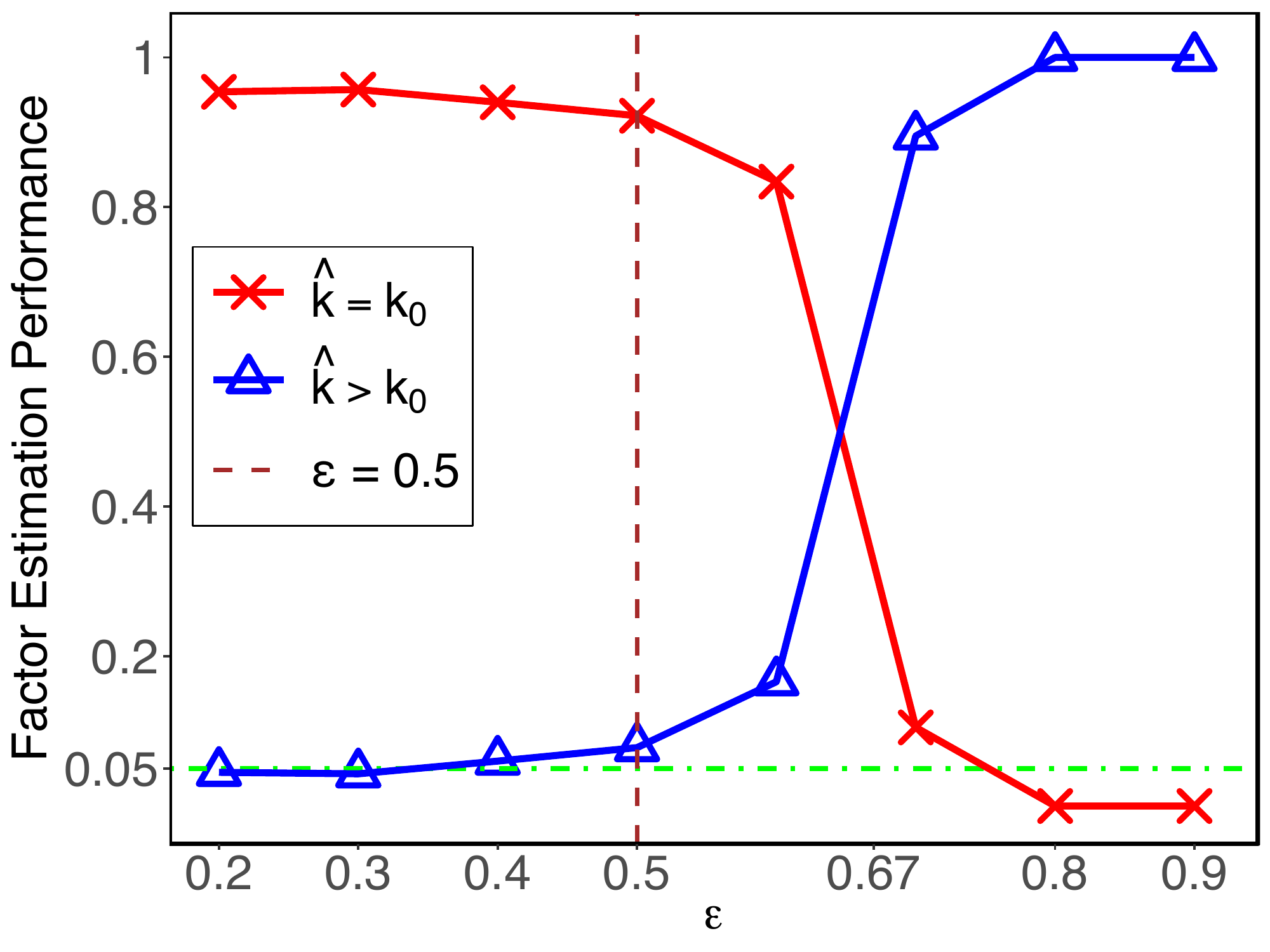}	
\end{subfigure}\quad \quad\ 
\begin{subfigure}[b]{0.4\textwidth}
\centering
\caption{\quad  $N=1000$; Bartlett Correction}
\includegraphics[width=\textwidth]{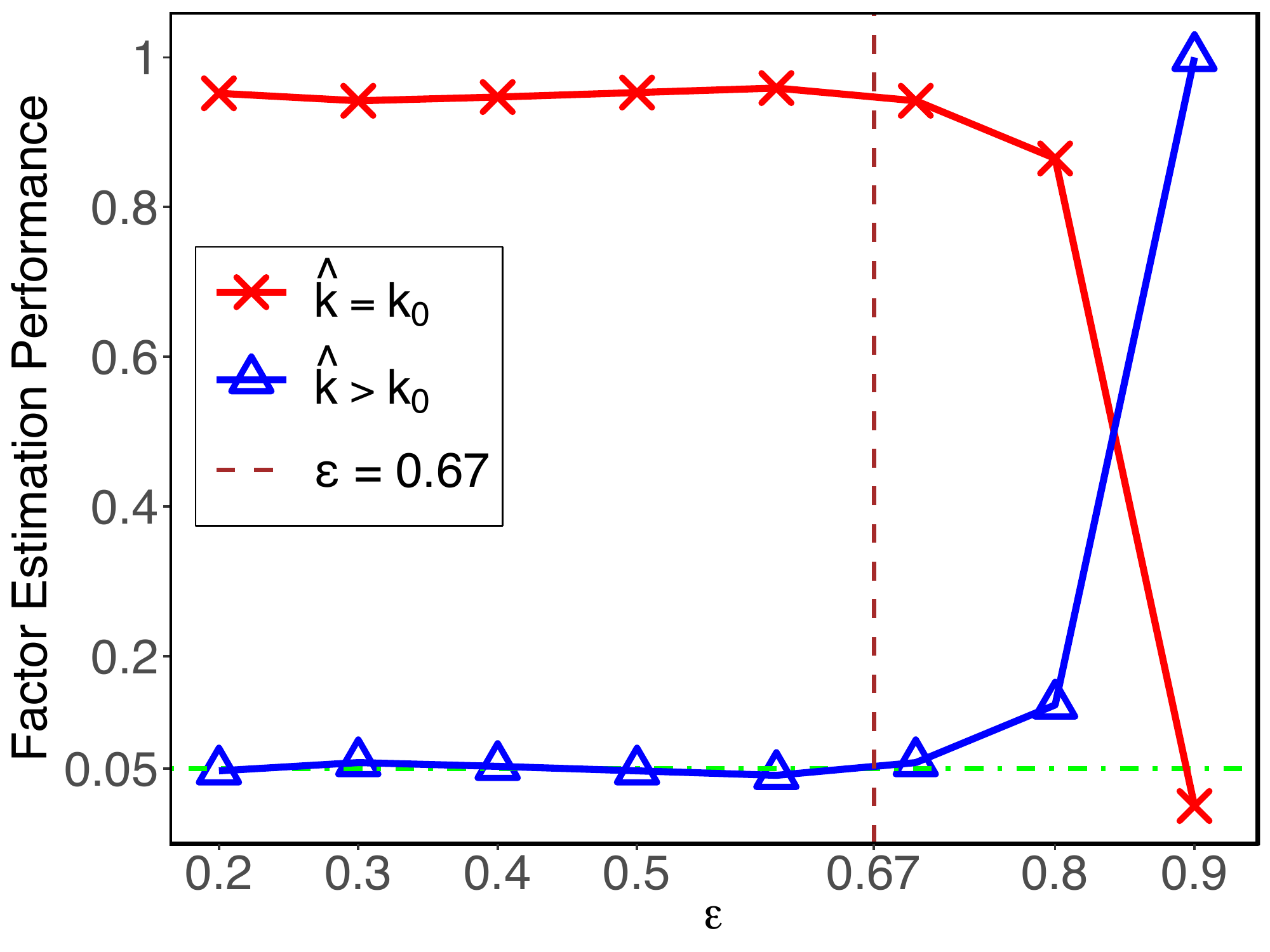}	
\end{subfigure}
\caption{Estimating the number of factors when $k_0=1$}\label{fig:factorsel1}	
\end{figure}

\begin{figure}
\captionsetup[subfigure]{labelformat=empty}
\centering
\begin{subfigure}[b]{0.4\textwidth}
\centering
\caption{\quad  $N=500$; No Correction}
\includegraphics[width=\textwidth]{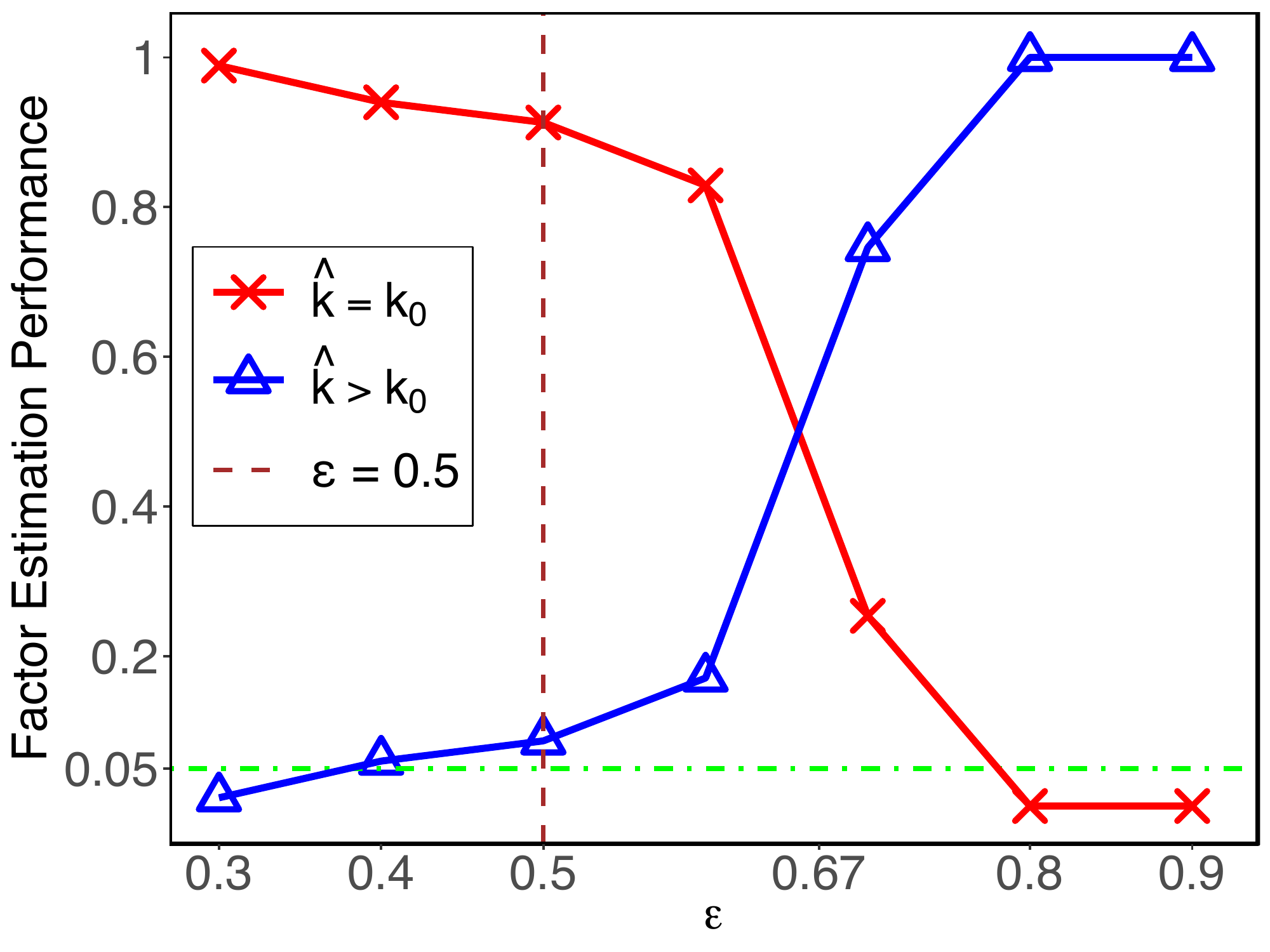}	
\end{subfigure}\quad \quad \
\begin{subfigure}[b]{0.4\textwidth}
\centering
\caption{\quad $N=500$; Bartlett Correction}
\includegraphics[width=\textwidth]{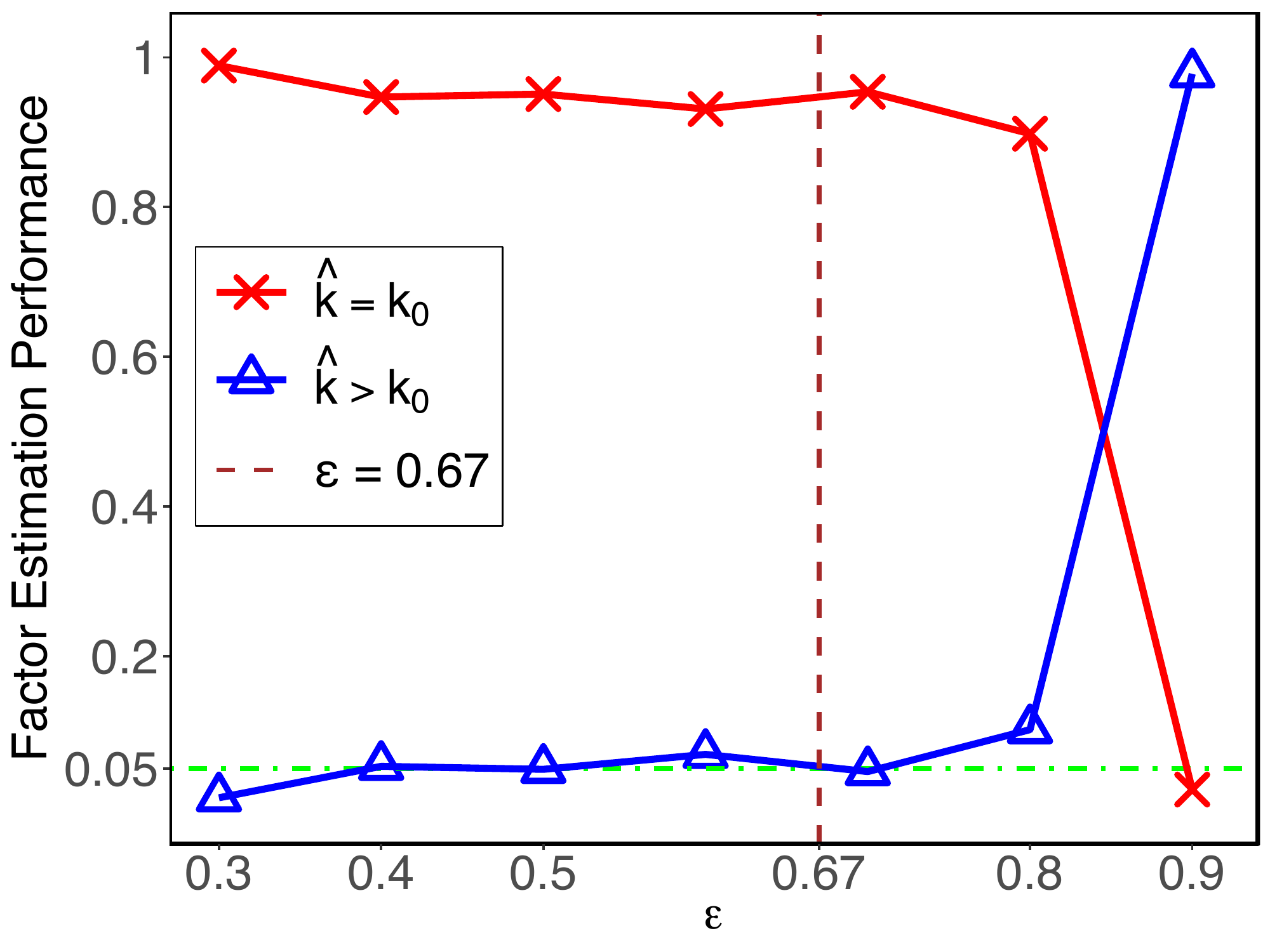}	
\end{subfigure}

\quad \\[-6pt]

\begin{subfigure}[b]{0.4\textwidth}
\centering
\caption{\quad $N=1000$; No Correction}
\includegraphics[width=\textwidth]{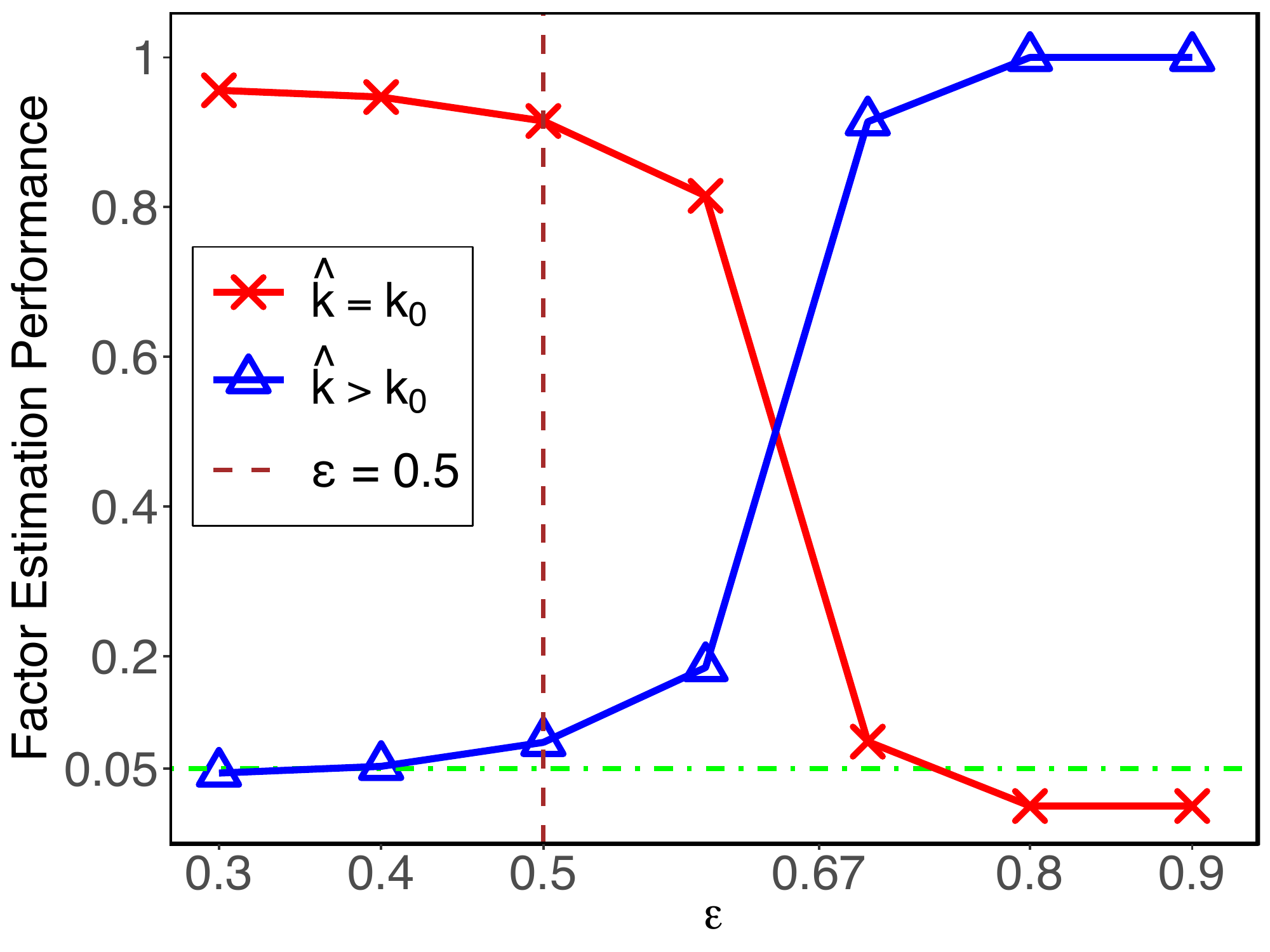}	
\end{subfigure}\quad \quad \
\begin{subfigure}[b]{0.4\textwidth}
\centering
\caption{\quad $N=1000$; Bartlett Correction}
\includegraphics[width=\textwidth]{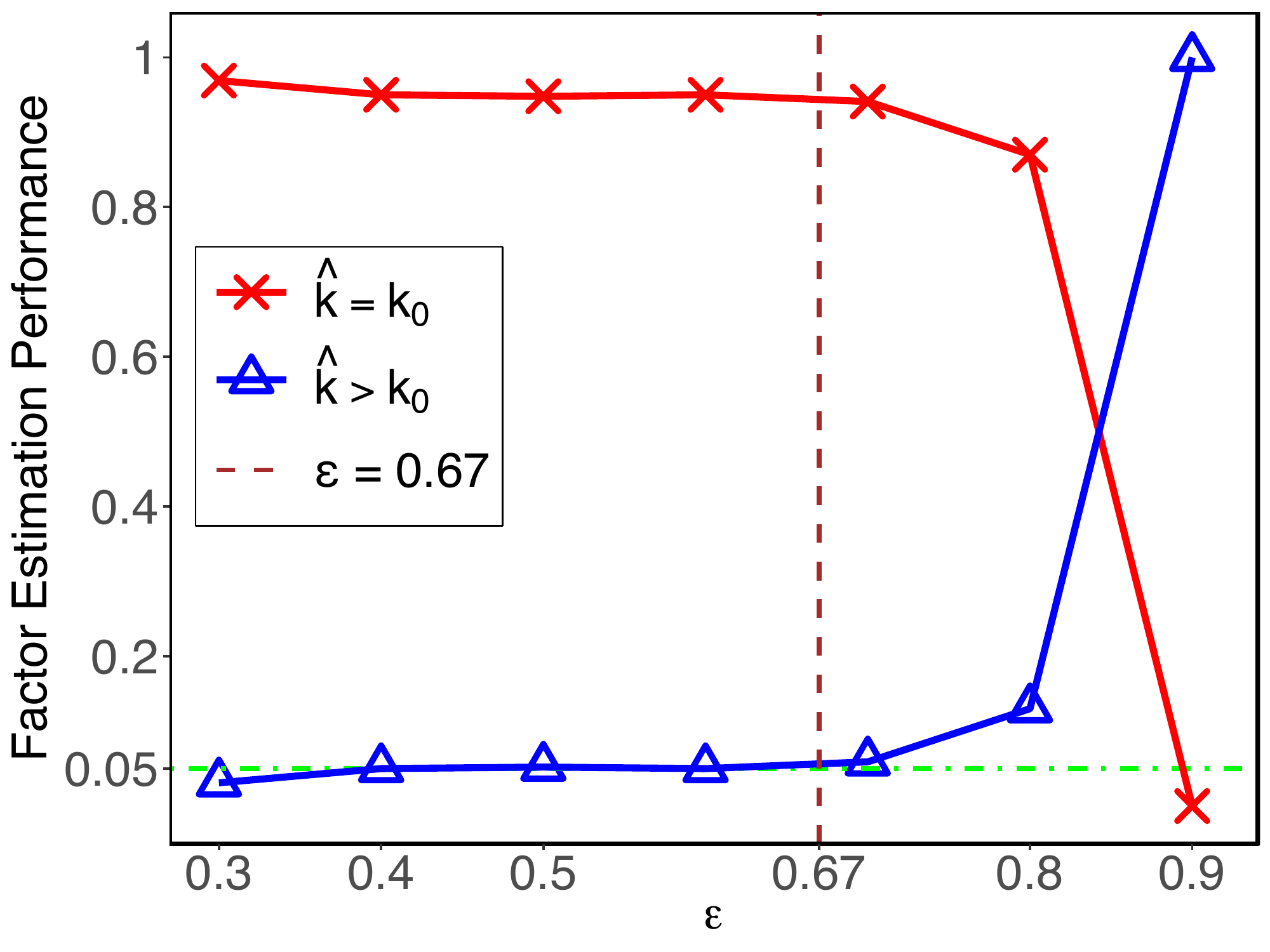}	
\end{subfigure}
\caption{Estimating the number of factors when $k_0=3$}\label{fig:factorsel2}	
\end{figure}

\end{document}